\documentclass[11pt]{article}

\usepackage[utf8]{inputenc}
\usepackage{titling}

\usepackage{xspace}
\usepackage{amsmath}
\usepackage{tikz-cd}
\usepackage{mathtools}
\usepackage{multicol}

\usepackage[ruled,linesnumbered,lined,commentsnumbered]{algorithm2e}

\usepackage[english]{babel}
\usepackage{amsmath,amsthm,amssymb,mathrsfs} 
\usepackage{ae,aecompl}
\usepackage{times}
\usepackage[pagebackref,hypertexnames=false]{hyperref}
\usepackage{fancybox,fancyhdr,graphics,epsfig}
\usepackage{textcomp}
\usepackage{authblk}
\usepackage{subfig}
\usepackage{accents}

\newtheorem{lem}{Lemma}[section]
\newtheorem{thm}[lem]{Theorem}
\newtheorem{prop}[lem]{Proposition}
\newtheorem{cor}[lem]{Corollary}

\newtheorem{con}[lem]{Conjecture}
\newtheorem{remark}[lem]{Remark}

\theoremstyle{definition}

\usepackage{amsmath,amsthm,bbm,amssymb}

\DeclareMathOperator{\rank}{rank}

\DeclareMathOperator{\vol}{Vol}

\newcommand{\curparam}[1]{{\left\{ #1 \right\}}}

\newcommand{\Rd}{\mathbb{R}^{d}}

\newcommand{\bs}{\boldsymbol}

\newcommand{\ind}{\mathbbm{1}}

\newcommand{\R}{{\mathbb{R}}}

\newcommand{\T}{{\mathbb{T}}}

\newcommand{\param}[1]{\left( #1 \right)}
\newcommand{\E}{\mathbb{E}}

\newcommand{\cX}{{\cal{X}}}
\newcommand{\cY}{{\cal{Y}}}

\newcommand{\cC}{{\cal{C}}}

\newcommand{\cA}{{\cal{A}}}

\newcommand{\cM}{{\cal{M}}}
\newcommand{\cP}{{\cal{P}}}

\newcommand{\cV}{{\cal{V}}}
\newcommand{\cB}{{\cal{B}}}
\newcommand{\cH}{{\cal{H}}}

\newcommand{\limninf}{\lim_{n\rightarrow\infty}}
\newcommand{\prob}{\mathbb{P}}
\newcommand{\eps}{\epsilon}
\newcommand{\mean}[1]{\E\left\{ #1 \right\}}
\newcommand{\set}[1]{\left\{ #1 \right\}}

\newcommand{\splitb}{\begin{split}}
\newcommand{\splite}{\end{split}}
\newcommand{\rmax}{{r_{\mathrm{max}}}}
\newcommand{\bx}{{\bs{x}}}

\newcommand{\vsimp}{{V_{\mathrm{simp}}}}

\newcommand{\hcrit}{h_{\mathrm{crit}}}
\newcommand{\gcrit}{g_{\mathrm{crit}}}

\newcommand{\headermath}[1]{\texorpdfstring{$#1$}{TEXT}}

\newcommand{\zint}{\bs{\mathrm{z}}}
\newcommand{\cI}{\mathcal{I}}
\newcommand{\vuni}{V_{\mathrm{uni}}}
\newcommand{\vdiff}{V_{\mathrm{diff}}}

\newcommand{\ubar}[1]{\underaccent{\bar}{#1}}

\newcommand{\Ncrit}[1]{N_{#1,r}^{k}}
\newcommand{\bth}{\bs{\theta}}
\newcommand{\Gr}{\mathrm{Gr}}
\newcommand{\vint}{V_{\mathrm{int}}}
\newcommand{\abs}[1]{|#1|}

\title{Sharp Phase Transitions for $k$-Fold Coverage \\Using Morse Theory}

\author[1]{Yohai Reani\thanks{syohai@campus.technion.ac.il}}
\author[2]{Omer Bobrowski \thanks{o.bobrowski@qmul.ac.uk}}

\affil[1]{Viterbi Faculty of Electrical and Computer Engineering\\Technion - Israel Institute of Technology}
\affil[2]{School of Mathematical Sciences \\Queen Mary University of London}

\date{}

\def\rconv{r_{\mathrm{conv}}}

\begin{document}

\maketitle

\begin{abstract}
We introduce a novel approach for studying random $k$-coverage, using Morse theory for the $k$-nearest neighbor ($k$-NN) distance function. We prove a sharp phase transition for the number of critical points of the $k$-NN distance function, from which we  conclude a phase transition for $k$-coverage. In addition, in the critical window our new framework enables us to prove a Poisson process approximation (in both location and size) for the last uncovered regions.
\end{abstract}

\section{Introduction}
The problem of random $k$-fold coverage can be described as follows: Given a fixed ``big set'' $X$ and a collection of random ``small sets'', sampled from some known distribution, what conditions on this distribution guarantee that every point in $X$ is covered by at least $k$ small sets ($k\ge 1$)? 
This problem arises in many fields, with potential applications in areas such as shape reconstruction  \cite{edelsbrunner2021multi, sheehy2012multicover}, wireless communication \cite{haenggi_stochastic_2009,wang_coverage_2011}, stochastic optimization \cite{zhigljavsky_stochastic_2007}, random topology \cite{dekergorlay2022random,bobrowski_vanishing_2017,bobrowski_random_2019}, genome analysis \cite{athreya_coverage_2004}, and boundary estimation \cite{cuevas_boundary_2004}. 

In this work we focus on the case where the big set is a closed $d$-dimensional Riemannian manifold $\cM$ and the small sets are geodesic balls of the same radius, centered randomly on $\cM$. In this setting, we let $B_r(x)$ denote the open ball centered at $x\in\cM$ with radius $r$. For a finite set of points $\cP\subset\cM$,  define the $k$-coverage process  as
\begin{equation}\label{eq:occupancy}
 B_r^{(k)}(\cP) := \{x\in\cM:|B_r(x)\cap\cP|\geq k\}.    
\end{equation}
In other words, $B_r^{(k)}(\cP)$ represents the set of points in $\cM$ that are covered by at least $k$ balls. Our goal is to study $B_r^{(k)}(\cP_n)$ for the case where $\cP_n$ is a homogeneous Poisson process on $\cM$ with rate $n$ as $n\rightarrow\infty$, and $r\coloneqq r(n)\rightarrow 0$. We express our results in terms of $\Lambda\coloneqq n\omega_d r^d$ -- the expected number of points in a $d$-dimensional ball of radius $r$, where $\omega_d$ is the volume of a $d$-dimensional unit ball. Our  main result is a phase transition for the $k$-coverage event, defined as 
\begin{equation}\label{eq:cov_event}
\cC_{r}^{(k)}:=\{\cM \subset B_{r}^{(k)}(\cP_n)\}. \end{equation}
We prove the following sharp phase transition (see Theorem \ref{thm:coverage_Td}),
\begin{equation}\label{eqn:coverage_pt}
\lim_{n\rightarrow\infty}\prob(\cC_{r}^{(k)})
=
\begin{cases}
1 & \Lambda = \log n + (d+k-2)\log\log n + w(n), \\
0 & \Lambda = \log n + (d+k-2)\log\log n - w(n),
\end{cases}
\end{equation}
for any $w(n)\to \infty$. In other words, there is a sequence of thresholds, increasing in $k$ via the second order term. It is important to note here, that such a phase transition can in fact be derived  from previous work, for example \cite{flatto_random_1977}.
The main novelty of the paper is rather in presenting a completely new {\bf Morse-theoretic approach for the proofs}.
So far, the main method for proving coverage has been to divide the big space in question into small regions, and to show that each region is covered (see \cite{penrose2023random}, for example). 
The approach we propose here is fundamentally different, and in some sense much more direct.
We use \emph{Morse Theory}, to argue that in the dense regime, near the coverage threshold, there is a one-one correspondence between the maxima of the $k$-NN ($k$-nearest neighbor) distance function and the uncovered regions. In \cite{reani2024knn} we showed that the maxima (as well as other critical points) have a  simple and localized characterization in terms of the Poisson process $\cP_n$. 
This characterization enables us to prove a sharp phase transition for the number of maxima in the dense regime, which then implies the $k$-coverage result. In addition, in the critical window, where $w(n)$ is constant, our approach unlocks a novel {\bf Poisson-process approximation for the last uncovered regions}, in the $(d+1)$-dimensional space of $\text{space}\!\times\! \text{size}$ (see Theorem \ref{thm:Poisson_lim}).

In this paper, we will phrase and prove the main results for the special case where $\cM=\T^d$ is a $d$-dimensional flat torus. This case significantly simplifies the calculations required. However, the work in \cite{bobrowski_topology_2014,bobrowski_random_2019} implies that exactly the same results will hold for smooth compact manifolds, and that the only gap in the proofs are the local metric estimates which will be the same as in \cite{bobrowski_topology_2014,bobrowski_random_2019}. We therefore choose to concentrate our efforts on the more elegant ``prototype'' setting of the torus.

In the course of proving the required results for the maxima, we are able to  prove general statements for critical points of any given index $\mu$ (see Theorem \ref{thm:prob_Fkr_cov}). These results are of an intrinsic interest, and they also enable us to make a couple of topological conclusions.  The first is about the expectation of the Euler characteristic of $B_r^{(k)}(\cP_n)$. The second is about \emph{homological connectivity} -- the event where the homology of $B_r^{(k)}(\cP_n)$ ``stabilizes'', and becomes isomorphic to the homology of $\cM$ (see Section \ref{sec:prelims} for an intuitive description of homology). In \cite{bobrowski2022homological}, detailed phase transitions were proven for homological connectivity in the case $k=1$. In this paper, we show upper bounds for the homological connectivity thresholds, which are an immediate corollary of our critical points analysis. Proving sharp phase transitions for homological connectivity is beyond the scope of this paper.

\vspace{5pt}
{\noindent \bf Related work.} The problem of random coverage has been studied extensively in many different settings 
\cite{flatto_random_1977,gilbert1965probability,glaz1979multiple,hall1985distribution,hall1985macroscopic,higgsCoveringOnePoint2025,holst1980lengths,holst1981convergence,janson_random_1986,penrose2023random,penroseRandomCoverageManifold2025,penroseRateConvergenceHallJanson2025,siegel1979asymptotic}.
Closest to the setting of this paper is the seminal work of Flatto and Newman \cite{flatto_random_1977}. Fixing a small radius $r$, they focused on the distribution of the smallest number of balls required to achieve $k$-coverage of a compact Riemannian manifold. Note the difference in setting -- fixed radius and random number of points in \cite{flatto_random_1977}, compared to deterministic and growing number of points here (via $n\to\infty$), while $r(n)\to 0$. Nevertheless, 
 Theorem 1.1 in \cite{flatto_random_1977} can be used to derive \eqref{eqn:coverage_pt}, as stated earlier. However, our novel Poisson approximation result (Theorem \ref{thm:Poisson_lim}) provides a much stronger statement than the results in \cite{flatto_random_1977}.
 In \cite{janson_random_1986}, these results were extended from balls to random convex sets. Replacing the Riemannian manifold setting with a compact subset of $\R^d$, similar results were presented in \cite{hall1985distribution,janson_random_1986,penrose2023random,penroseRandomCoverageManifold2025}. 
The work of Hall \cite{hall1985distribution} goes beyond estimating coverage probabilities, aiming to characterize the shape of the uncovered region. Most recently, Penrose \cite{penrose2023random} studied $R_{n,k}$ -- the smallest radius needed for $k$-coverage with $n$ balls, providing tighter convergence rates. 

As stated earlier, the main novelty in our paper is the use of Morse Theory to address the $k$-coverage problem. Due to the nature of the critical points, they provide the most precise way to analyze the vacancy process, representing the exact location of the vacancy components before they vanish. The Morse theoretic approach to coverage ($k=1$) was presented in \cite{bobrowski2022homological}, as a direct consequence of the study of homological connectivity in random \v Cech complexes. The Morse function used there was simply the distance function. Here we show  that this approach generalizes to any $k\ge 1$, by taking the $k$-NN distance function.
We note that the applicability of the Morse theoretic approach we present here is restricted to the ball covering setting, and it remains future work to study whether similar approaches can be applied for other covering objects.

\section{Preliminaries}\label{sec:prelims}

In this section we briefly introduce the main geometric/topological objects studied in this paper. In particular, we introduce Morse theory for the $k$-NN distance function, which is the main tool we use in this paper.

\subsection{Homology} Homology is a topological-algebraic structure that describes the shape of a topological space via its connected components, holes, cavities, and their higher-dimensional generalizations. Formally, for a topological space $M$, homology is represented by a sequence of abelian groups, $H_0(M),H_1(M),\ldots$, where the zeroth group $H_0(M)$ is generated by the connected components of $M$ (also known as $0$-cycles), $H_1(M)$ by  ``holes'' ($1$-cycles), $H_2(M)$ by ``voids'' or ``cavities'' ($2$-cycles), and $H_i(M)$ by $i$-dimensional cycles (an $i$-cycle can be thought of as the boundary of a $(i+1)$-dimensional body).
The ranks of the homology groups are known as the \emph{Betti numbers}, defined as
$\beta_i(M) = \rank(H_i(M))$, see Figure  \ref{fig:hom_ex} for examples .
For formal definitions of homology and more details see \cite{hatcher_algebraic_2002}.   

\begin{figure}
    \centering
    \includegraphics[width=0.8\textwidth]{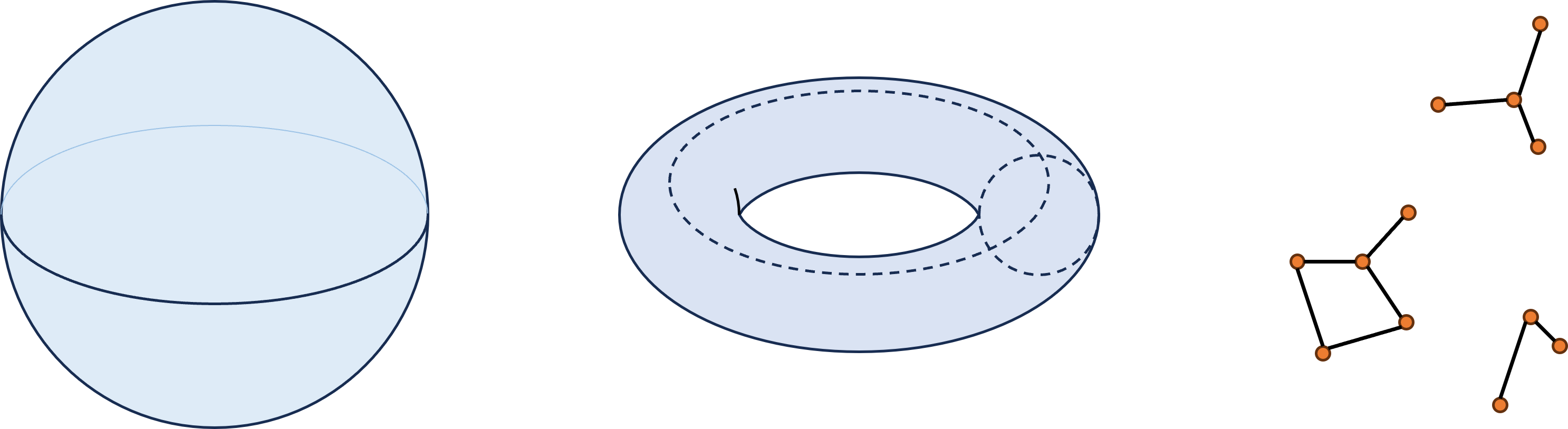}
    \caption{Homology. 
    Left: The 2-dimensional sphere $\mathbb{S}^2$ has a single connected component ($0$-cycle), enclosing an ``air pocket" ($2$-cycle). Hence, $\beta_0(\mathbb{S}^2)=\beta_2(\mathbb{S}^2)=1$, and $\beta_i(\mathbb{S}^2)=0$ for all $i\neq 0,2$. 
    Center: The torus $\T^2$ has one connected component, two independent $1$-cycles (dashed lines) and a single $2$-cycle. Hence, $\beta_0(\T^2)=1,\beta_1(\T^2)=2,\beta_2(\T^2)=1$.
    Right: A planar graph $G$ on $12$ vertices, with three connected components and a single $1$-cycle. Hence, $\beta_0(G)=3$ and $\beta_1(G)=1$.}
    \label{fig:hom_ex}
\end{figure}
\subsection{Morse theory} 
Morse theory studies the topology of manifolds via smooth functions defined on them. 
More specifically, Morse theory provides a  link between the critical points of a  function and changes in the topological structure of the underlying manifold. One of the powerful properties of Morse theory is that it enables us to draw global topological conclusions about structure, from local critical point configurations. 

Formally, let $\cM$ be a manifold. A function $f:\cM\rightarrow\R$ 
is a \emph{Morse function} if it is smooth and all its critical points are non-degenerate. Each critical point of $f$ is associated with an  \emph{index}, which is the number of negative eigenvalues of the Hessian of $f$ at the critical point. Denote by $\cM_r=f^{-1}((-\infty, r])=\{x\in \cM:f(x)\leq r\}$  the sublevel sets of $f$, and 
let $I = (a,b)$ be a non-empty interval. By Morse Theory, assuming that $(b-a)$ is sufficiently small, there are two possible cases: (1) If $I$  does not contain any critical value of $f$, then $H_i(\cM_a) \cong H_i(\cM_b)$ for all $i$ (more precisely, $\cM_a$ and $\cM_b$ are homotopy equivalent). In other words, the homology does not change as we increase the level from $a$ to $b$.
(2) If $I$ contains a single critical value associated to a critical point of index $\mu$,
then exactly one of the following holds,
\[
\beta_{\mu}(\cM_b) = \beta_{\mu}(\cM_a) + 1,
\quad\text{or}\quad
\beta_{\mu-1}(\cM_b) = \beta_{\mu-1}(\cM_a) - 1,
\]
in other words, increasing the level from $a$ to $b$, either a new $\mu$-cycle is generated, or an existing $(\mu-1)$- is terminated. For more details on Morse theory, see \cite{milnor_morse_1963}.

\subsection{The $k$-NN distance function} 
Let $\cP$ be a finite subset of $\cM$, with $|\cP|\ge k$. We define the $k$-nearest neighbor ($k$-NN) distance function $d^{(k)}_{\cP}:\cM\to \R^+$  as 
\[
d_{\cP}^{(k)}(x) := \min\set{r : |B_r(x)\cap \cP| \ge k},
\]
Our interest in $d^{(k)}_{\cP}$ comes from fact that its sub-level sets satisfy
\[(d_{\cP}^{(k)})^{-1}((-\infty, r]) = B_{r}^{(k)}(\cP),\]
where $B_{r}^{(k)}(\cP)$ is the $k$-coverage process \eqref{eq:occupancy}.
Note that $d_{\cP}^{(k)}$ is not a smooth function, and in particular its critical points cannot be defined using derivatives. Hence, classical Morse theory cannot be applied to it directly. Instead, we will use the framework developed in  \cite{reani2024knn}, which characterizes the critical points of $d_{\cP}^{(k)}$ and their effect on the homology of $B_{r}^{(k)}(\cP)$.

\subsection{Critical points of the $k$-NN distance function} In the following, we use the definitions of critical points from \cite{reani2024knn}, which were adapted from the more general framework for piecewise smooth functions developed in \cite{agrachev1997morse}. 
For simplicity, we will phrase the following definitions for the case where $\cM=\Rd$, but these can be easily adapted to the general smooth Riemannian manifolds as we discuss in Section \ref{sec:torus} (see also \cite{bobrowski_random_2019}). In addition, from now on we will assume that the points of $\cP$ are in general position, i.e., no subset of $\cP$ of size $m+1$ lies on a $(m-1)$-dimensional hyperplane ($1\leq m\leq d$). Note that for a random Poisson process, this assumption holds almost surely.
  
For each critical point of $d_{\cP}^{(k)}$ we associate a \emph{critical configuration} of points in $\cP$, as follows. Let $\cX\subset\cP$ 
where $2\leq |\cX|\leq d+1$, and denote by $S(\cX)$ the $(|\cX|-2)$-dimensional unique circumsphere of $\cX$. Denote \begin{equation}\label{eq:crit_ball}
\begin{split}
    c(\cX)&:= \text{The center of }S(\cX),
    \\
    \rho(\cX)&:=\text{The radius of }S(\cX),
    \\
    \cB(\cX)&:= \text{The }d\text{-dimensional open ball with center }c(\cX)\text{ and radius }\rho(\cX),
    \\
    \sigma(\cX)&:=\text{ the }(|\cX|-1)\text{-dimensional open simplex spanned by }\cX,
    \end{split}
\end{equation}
and
\begin{equation}\label{eq:I_and_mu}
    \begin{split}
        \cI(\cX,\cP) & := |\cB(\cX)\cap\cP|,
        \\
        \mu(\cX,\cP) & := |\cX| + \cI(\cX,\cP) - k.
    \end{split}
\end{equation}

\begin{thm} \label{thm:crit_pts}[Theorem 1 in \cite{reani2024knn}] 
\label{def:crit_pnt}
The point $c=c(\cX)$ is a critical point of $d_{\cP}^{(k)}$ of index $\mu_c=\mu(\cX,\cP)$, if and only if $c\in\sigma(\cX)$ and $k-|\cX| \le \cI(\cX,\cP) \le k-1$.
\end{thm}

Note that Theorem \ref{thm:crit_pts} implies that $0\le \mu(\cX,\cP) \le d$.
Also note that $c\in \sigma(\cX)$ implies that $\cX$ cannot  be contained in one hemisphere of $S(\cX)$.
See Figure \ref{fig:morse_regular} for examples.

\begin{figure}
    \centering
    \includegraphics[width=\textwidth]{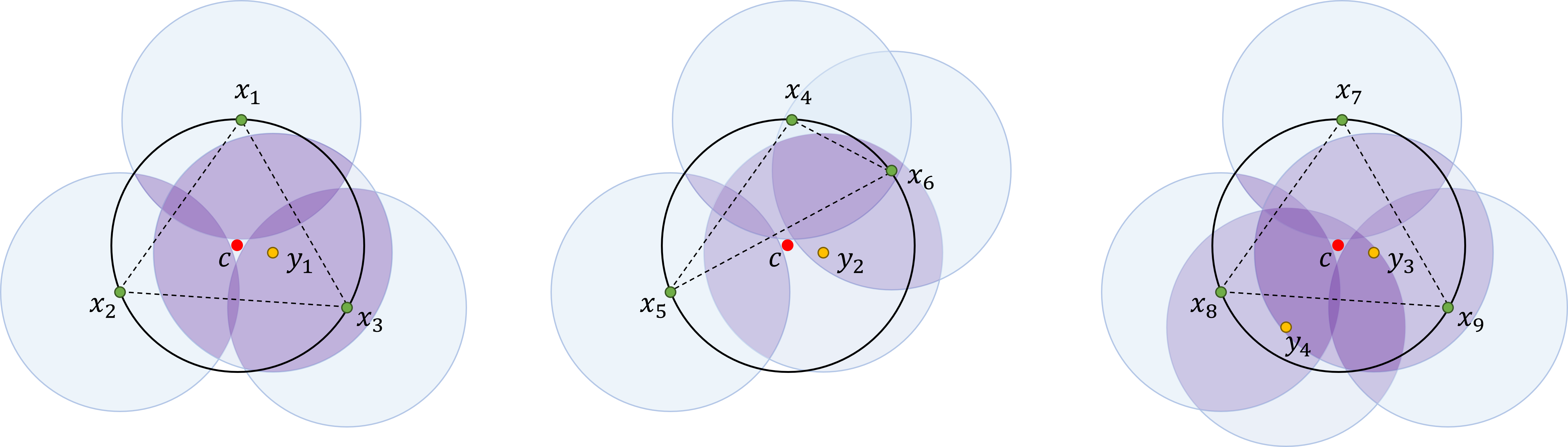}
    \caption{Critical points of $d_{\cP}^{(k)}$ in $\R^2$, for $k=2$. Left: The set $\cX_1=\{x_1,x_2,x_3\}$ induces a  critical point $c$ of index $\mu = 2$, since the interior of $\cB(\cX_1)$ contains exactly a single point $y_1$, and $\sigma(\cX_1)$ (dashed triangle) includes $c$. The shaded purple region is $B_r^{(2)}(\cP)$. Center: the set $\cX_2=\{x_4,x_5,x_6\}$ does not induce a  critical point, since $c\not\in\sigma(\cX_2)$. Right: the set $\cX_3=\{x_7,x_8,x_9\}$ does not induce a  critical point, since the interior of $\cB(\cX_3)$ includes more than one point.}
    \label{fig:morse_regular}
\end{figure}

\begin{remark}
While the critical points of $d_{\cP}^{(k)}$ can have indexes ranging from $0$ to $d$ (as in the smooth function case), when we analyze the $k$-coverage phase transition (Theorem \ref{thm:coverage_Td}), we only need to consider critical points of indexes $d$ (maxima), as explained in Section \ref{sec:coverage}.  For the Poisson approximation result (Theorem \ref{thm:Poisson_lim}) we will also need to consider critical points of index $d-1$ (see Lemma \ref{lem:oneone}).
Nevertheless, our phase transition results for critical points will be proved for all indexes.
\end{remark}

Once a critical point is reached, the homology of the sublevel sets of $d_{\cP}^{(k)}$ changes. In particular, in \cite{reani2024knn} it is shown that a critical point of index $\mu$ can induce changes (potentially more than one) in the $(\mu-1)$-th homology (by eliminating cycles) and the $\mu$-th homology (by generating cycles). The total number of changes (generations and eliminations) induced by a single critical point $c=c(\cX)$ of index $\mu_c$, denoted by $\Delta_c$, is given by
\begin{equation}\label{eq:delta_c}
    \Delta_c=\binom{|\cX|-1}{\mu_c}.
\end{equation}
For more details on the Morse Theory for the $k$-NN distance function see \cite{reani2024knn}.

\subsection{The flat torus}\label{sec:torus}
The flat torus, denoted by $\T^d$, is a compact manifold, that is homogeneous, and locally Euclidean. It is defined as the quotient $\T^d\coloneqq\R^d/\mathbb{Z}^{d}$, equipped with the toroidal metric $d_{\T}(x,y)=\min_{\delta\in\mathbb{Z}^{d}} \|x-y+\delta\|$, for all $x,y\in[0,1]^d$. Alternatively, it can be viewed as the unit box in $d$-dimensions $[0,1]^{d}$ with periodic boundary conditions.

The advantages of studying the flat torus 
are that (a) its metric is locally Euclidean, simplifying geometric arguments, while (b) there are no boundary effects to consider. As mentioned earlier, while our main results are stated for $\cM = \T^d$, similar analysis to the one in \cite{bobrowski_random_2019} can be used to generalize them to smooth Riemannian manifolds. Thus, $\T^d$ serves as a valid ``prototype manifold''.

\noindent{\bf The $k$-NN distance function on the flat torus.}
The periodicity of $\T^d$ imposes certain constraints on the assumptions on the $k$-NN distance function and the characterization of its critical points. These constraints have no impact on our results, since our analysis only considers small neighborhoods. Nevertheless, we wish to briefly discuss them here.

The radius of convexity $\rconv$ of a Riemannian manifold $\cM$ is defined as the maximal radius such that any ball of radius $r<\rconv$ in $\cM$ is convex. That is, the geodesic between any pair of points in the ball lies entirely inside the ball. 
For $\T^d$ we have $\rconv=1/4$, since any ball with radius $r \geq 1/4$ contains pairs of points whose geodesic lies partially outside the ball, rendering it non-convex.

The definition of critical points in \cite{reani2024knn} is based on the observation that $\delta_{\cP}^{(k)}\coloneqq(d_{\cP}^{(k)})^2:\Rd\rightarrow\R$ (the squared $k$-NN distance function) can be represented as a \emph{continuous selection} of the set of smooth functions $\{d_p^2(x)\coloneqq\|p-x\|^2:p\in\cP\}$, namely, at each $x\in\Rd$ there exits $p\in\cP$ such that $\delta_{\cP}^{(k)}(x)=d_p^2(x)$.
While this is sufficient in $\Rd$, a more refined treatment is required in $\T^d$. In particular, $d_p^2$ is no longer smooth due to the periodic boundary conditions. 
Nevertheless, the restriction of $\delta_{\cP}^{(k)}$ to $B_{r}^{(k)}(\cP)$, where $r<\rconv$, is a continuous selection of smooth functions, as we state in the following lemma.
\begin{lem}\label{lem:CS_in_Td}
Let $\cP\subset\T^d$ be a finite set, and denote $\rmax = \rconv$. If $r<\rmax$, then, the restriction $\delta_{\cP}^{(k)}:B_{r}^{(k)}(\cP)\rightarrow\R$ is a continuous selection of smooth functions.
\end{lem}

\begin{proof}
Note that the restriction of $d_{p}^2$ to $B_r(p)$, with $r<\rmax$, is smooth. However, this smoothness does not generally extend to the entire region $B_r^{(k)}(\cP)$. 
Nevertheless, for $x\in B_{r}^{(k)}(\cP){\setminus}B_{r}(p)$, we have $d_{p}^2(x)>\delta_{\cP}^{(k)}(x)$, which implies that $\delta_{\cP}^{(k)}=\delta_{\cP{\setminus}\{p\}}^{(k)}$ in this region. In other words, $d_p^2$ has no effect outside $B_r(p)$. 
This allows us to replace each $d_{p}^2$ with a smooth alternative function $\tilde{d}^2_{p}(x)$, such that $\tilde{d}^2_{p}(x)=d_{p}^2(x)$ for all $x\in B_{\rmax}(p)$. Using this alternative representation does not change the values of $\delta_{\cP}^{(k)}$, and shows that it is indeed a continuous selection.
\end{proof}

Another challenge arises in the definition of critical points in $\T^d$. Let $\cX\subset\cP$ of size $m+1$, where $1\leq m \leq d$. While for $\Rd$, the $(m-1)$-circumsphere $S(\cX)$ is uniquely defined, in $\T^d$ this is not the case due to its periodicity. Therefore, we define $S(\cX)$ here as the smallest circumsphere of $\cX$. The restriction 
$r<\rmax$ guarantees that we only consider sets $\cX$ with diameter smaller than $2\rmax$, in which case the sphere $S(\cX)$ is unique.
Note that the entire analysis in this paper deals with radii much smaller than $\rmax$ ($r=r(n)\to 0$), and therefore the restricting $r < \rmax$ have no effect on our results.

\subsection{Notations} We use the following asymptotic notations throughout the paper. We use $a_n\approx b_n$ to denote  $\lim_{n\rightarrow\infty}\frac{a_n}{b_n}=1$, and $a_n\sim b_n$ to denote $\lim_{n\rightarrow\infty}\frac{a_n}{b_n}=C$, where $C<\infty$ is a nonzero constant that does not depend on $n$. 

\section{Main results}\label{sec:main_results}
In this section, we present the main contributions of this paper. These include:
(a) A sharp phase transition for $k$-coverage; (b) A Poisson limit for last vacant regions;  (c) Phase transitions for the critical points of $d_{\cP}^{(k)}$ of varying indexes; (d) Topological conclusions.

\subsection{Phase transition for $k$-coverage}
Let $\cP_n$ be a  homogeneous Poisson point process in $\cM = \T^d$.
We consider the random $k$-coverage process induced by $B_{r}^{(k)}(\cP_n)$ \eqref{eq:occupancy}. To simplify notation, from now on we use $B_r^{(k)}\coloneqq B_{r}^{(k)}(\cP_n)$.
Recall that we focus on the asymptotic regime where $n\rightarrow\infty$ and $r=r(n)\rightarrow 0$, and we express our results in terms of $\Lambda=n\omega_d r^d$.
Our first result is a phase transition for $k$-coverage event $\cC_{r}^{(k)}:=\{ \T^d\subset B_{r}^{(k)} \}$.

\begin{thm}\label{thm:coverage_Td}
Let $k\geq 1$, and suppose that $w(n)\rightarrow\infty$ as $n\rightarrow\infty$. Then,
\begin{equation*}
\lim_{n\rightarrow\infty}\prob(\cC_{r}^{(k)})
=
\begin{cases}
1 & \Lambda = \log n + (d+k-2)\log\log n + w(n), \\
0 & \Lambda = \log n + (d+k-2)\log\log n - w(n).
\end{cases}
\end{equation*}
\end{thm}
Note that the thresholds have the same leading term, i.e., they occur almost simultaneously, and they differ only in the second order term. As expected, the thresholds appear in an ascending order in $k$ (since $(k+1)$-coverage implies $k$-coverage).
In addition, note that $k=1$ recovers simple coverage, i.e.~the union of balls covers $\T^d$ (cf. Theorem 3.1 in \cite{bobrowski2022homological}). 
As stated earlier, these phase transitions can be derived from previous results, for example Theorem 1.1 in \cite{flatto_random_1977}. The main novelty here is our Morse theoretic proof for this phenomenon.

\subsection{A Poisson-process limit in the critical window}

Here, we consider the critical window
\begin{equation}\label{eq:crit_win}
    \Lambda \coloneqq n\omega_d r_0^d=\log n +(d+k-2)\log\log n +\lambda_0,
\end{equation}
where $\lambda_0\in\R$ is fixed. 
Here, prior to coverage, we are interested in the 
distribution of the last uncovered regions. Define the \emph{$k$-vacancy process} as the complement of the $k$-coverage process, i.e., 
\begin{equation}\label{eq:Vacancy}
\cV_r^{(k)} = \cV_{r}^{(k)}(\cP_n)\coloneqq \T^d \backslash B_r^{(k)}(\cP_n).
\end{equation}
Our goal is to study the distribution of the last connected components of $\cV_{r}^{(k)}$ that ``survive'' into the critical window, before they vanish.
We will show in Lemma \ref{lem:oneone}
that there is one-to-one correspondence between these components and the maxima of the $k$-NN distance function $d_{\cP_n}^{(k)}$. Thus, we
can associate to every vacant component a pair of values $(c,\rho)$, that are the critical point and critical value. These correspond to the location of the vacant component right before it vanishes and the radius at which it does. Also note that if $r < \rho$ (but still within the critical window), then the vacant component in $\cV_{r}^{(k)}$ is bounded by the ball $B_{\rho}(c)
$.

Let $\cC^d = \cC^d(\cP_n)$ be the set of all critical points of index $d$ (maxima) generated by $\cP_n$. For a critical point $c\in \cC^d$ denote its critical value by $\rho_c\coloneqq d_{\cP}^{(k)}(c)$, and set
\begin{equation}\label{eq:lambda_c}
\lambda_c \coloneqq n\omega_d \rho_c^d-\log n -(d+k-2)\log\log n.
\end{equation}
Next, define a point process on $\T^d\times\R$ by 
\begin{equation}\label{eq:xi_k}
    \xi_k = \xi_k[\cP_n] \coloneqq \sum_{c\in\cC^d}
    \ind\{\rho_c\in(r_0,\sqrt{r_0}\hspace{2pt} ]\}\delta_{(c,\lambda_c)},
\end{equation}
where $\delta_{x}$ is the Dirac delta measure, and $r_0$ is defined via \eqref{eq:crit_win}. 
Note that $\xi_k$ is a map from point processes in $\T^d$, to point processes (measures) in $\mathbb{Y}\coloneqq \T^d\times \R_0$, where $\R_0=[\lambda_0,\infty)$. We will think of this process as representing the last vacant connected components. The choice of the interval $(r_0,\sqrt{r_0}]$ is for technical reasons in the proofs. However, since $\sqrt{r_0} \gg r_0$, we are in practice considering all relevant critical points.

The following result states that $\xi_k$ converges to a Poisson process in $\mathbb{Y}$ under the Kantorovich-Rubinstein (KR) distance,  defined as 
\[
d_{\mathrm{KR}}(\xi, \zeta) = \sup_{h \in \mathrm{LIP}(\mathbb{Y})} \left| \E(h(\xi)) - \E(h(\zeta)) \right|,
\]
where $\mathrm{LIP}(\mathbb{Y})$ is the class of measurable 1-Lipschitz functions on $\mathbb{Y}$.
We note that the KR distance upper bounds the total variation (TV) distance.

\begin{thm}\label{thm:Poisson_lim}
Let $\lambda_0\in\R$, and $\xi_k$ as defined above. Then, for $n\geq 3$, we have
\[
d_{\mathrm{KR}}(\xi_k,\zeta_k)\leq C_{\lambda_0}(\log\log n)^{d+k}(\log n)^{-\frac{1}{2(d+k)}},
\]
for some $C_{\lambda_0}>0$, where $\zeta_k$ is a Poisson process on $\mathbb{Y}=\T^d\times\R_0$ with intensity $C_d e^{-\lambda}d\lambda dc$, and where $C_d>0$ is a constant defined in \eqref{eq:C_d}.

In particular, this implies $\xi_k\xrightarrow[]{\mathrm{KR}}\zeta_k$ as $n\rightarrow\infty$.
\end{thm}

Let $V_{k,r}$ denote the number of vacant connected components at radius $r$.
The Poisson limit in Theorem \ref{thm:Poisson_lim} implies the following.

\begin{cor}
    Let $\Lambda = \log n + (d+k-2)\log\log n + \lambda_0$, for some $\lambda_0\in \R$. Then,
    \[
        V_{k,r} \xrightarrow[]{\mathrm{TV}} \mathrm{Poisson}(C_d e^{-\lambda_0}),
    \]
    which, in particular, implies that
    \[
        \limninf\prob(\cC_r^{(k)}) = e^{-C_de^{-\lambda_0}}.
    \]
\end{cor}

\subsection{Critical points of the $k$-NN distance function}

The key advantage of our Morse theoretic approach, is that it allows us to replace the ``brute-force'' search for uncovered regions, with a combinatorial counting of critical points. To this end, define,
\begin{equation}\label{eqn:N_k}
\Ncrit{\mu} \coloneqq \text{\#critical points $c\in\T^d$ of index }\mu\text{ with }\rho_c = d_{\cP}^{(k)}(c)\geq r.    
\end{equation}

The next result presents phase transitions for the vanishing of the critical points of the $k$-NN distance function. This result is important in its own right, while also serving as an essential component in the proof of Theorem \ref{thm:coverage_Td}.

\begin{thm}\label{thm:prob_Fkr_cov}
Let $k\geq 1$ and $0\leq \mu\leq d$, and suppose that $w(n)\rightarrow\infty$ as $n\rightarrow\infty$. Then,
\[
\lim_{n\rightarrow\infty}\prob(\Ncrit{\mu}=0)
=
\begin{cases}
1 & \Lambda = \log n + (\mu+k-2)\log\log n + w(n), \\
0 & \Lambda = \log n + (\mu+k-2)\log\log n - w(n),
\end{cases}
\]
excluding the case $k=1$, $\mu=0$, in which $N_{k,r}^\mu = 0$ for all $r>0$.
\end{thm}

\subsection{Topological conclusions}
The Euler characteristic is an integer-valued topological invariant. For a topological space $X$ it can be defined as the signed-sum of the Betti numbers,
\[
\chi(X) := \sum_{i=0}^{\infty}(-1)^{i}\beta_i(X).   
\] 
However, one of the outcomes of Morse Theory is that
\[
\chi(X) = \sum_{i=0}^{\infty}(-1)^{i}N_i(X),
\]
where $N_i(X)$ is the number of critical points of index $i$ for any Morse function $f:X\to \R$. Our analysis of critical points for the $k$-NN distance function  enables us to prove the following.
\begin{prop}\label{prop:euler_char}
If $r<\rmax$ then
\[
\E\{\chi(B_r^{(k)})\} = ne^{-\Lambda}
\sum_{i=0}^{d+k-2}A_i\Lambda^i,
\]
where $A_i$ are constants that depend only on $i$, $k$ and $d$ , and are defined in \eqref{eq:euler_consts}. Additionally, $A_0=1$ for $k=1$, and $A_0=0$ for $k>1$.
\end{prop}

Note that this result is non-asymptotic,  generalizing the result for $k=1$ in  \cite{bobrowski_vanishing_2017}. In addition, for $r\ge\rmax$, $k$-coverage is reached (with high probability), and we have $\E\{\chi(B_r^{(k)})\}\approx\chi(\T^d)=0$, which agrees asymptotically with the formula in Proposition \ref{prop:euler_char}.

\vspace{5pt}
The last topic we want to discuss is \emph{homological connectivity}. 
Upon $k$-coverage, we have $B_r^{(k)} = \T^d$
and consequently $H_i(B_r{(k)})=H_i(\T^d)$ for all $i$.
However, slightly before  $k$-coverage, $B_r^{(k)}$ might have an intricate structure, and in particular $H_i(B_r{(k)})\ne H_i(\T^d)$. Here, we  seek the smallest $r$ that guarantees the ``stabilization'' of $H_i$ in the sense that if we further increase $r$, $H_i(B_r^{(k)})$ will not change anymore, and in particular it will be isomorphic to $H_i(\T^d)$.
More concretely, we denote the \emph{homological connectivity} event as,
\begin{equation}\label{eq:hom_con_event}
\cH_{i,r}^{(k)}\coloneqq\left\{H_{i}(B_{s}^{(k)})\cong H_{i}(\T^d),\ \forall s\geq r\right\},   
\end{equation}
where the isomorphism is induced by the inclusion $B_r^{(k)}\hookrightarrow \T^d$.
Our analysis of critical points, including Theorem \ref{thm:prob_Fkr_cov}, leads to the following.

\begin{cor}\label{cor:hom_con}
Let $k\geq 1$, and suppose that $w(n)\rightarrow\infty$ as $n\rightarrow\infty$. Then,
\[
\lim_{n\rightarrow\infty}\prob(\cH_{d-1,r}^{(k)})
=
\lim_{n\rightarrow\infty}\prob(\cH_{d,r}^{(k)})
=
\begin{cases}
1 & \Lambda = \log n + (d+k-2)\log\log n + w(n), \\
0 & \Lambda = \log n + (d+k-2)\log\log n - w(n).
\end{cases}
\]    
In addition, for any $0\le i \le d-2$, if $\Lambda = \log n + (i+k-1)\log\log n + w(n)$, then
\[
\limninf\prob({\cH_{i,r}^{(k)}}) = 1.
\]
\end{cor}

Note that the thresholds for $i=d-1,d$ are the same as the $k$-coverage threshold. For $i\le d-1$ Corollary \ref{cor:hom_con} provides only an upper bound, but based on \cite{bobrowski2022homological} we conjecture that the following is true.

\begin{con}\label{con:hom_con}
Let $k\geq 1$ and $0\leq i\leq d-2$, and suppose that $w(n)\rightarrow\infty$ as $n\rightarrow\infty$. Then,
\[
\lim_{n\rightarrow\infty}\prob(\cH_{i,r}^{(k)})
=
\begin{cases}
1 & \Lambda = \log n + (i+k-2)\log\log n + w(n), \\
0 & \Lambda = \log n + (i+k-2)\log\log n - w(n).
\end{cases}
\]  
\end{con}
Note that Conjecture \ref{con:hom_con} and Corollary \ref{cor:hom_con} are generalization of the $k=1$ case, that was proved as Theorems 3.1 and 3.2 in \cite{bobrowski2022homological}. The main gap between Theorem \ref{thm:prob_Fkr_cov} and  Conjecture \ref{con:hom_con} is proving separate phase transitions for positive critical points (those generating cycles) and negative critical points (those terminating cycles). For $k=1$ this was done in \cite{bobrowski2022homological}, but the analysis there does not generalize naturally to $k>1$, and thus proving Conjecture \ref{con:hom_con} remains as future work.

\section{Proofs}
\subsection{Critical points of the \headermath{k}-NN distance function}
In this section we provide the proof for Theorem \ref{thm:prob_Fkr_cov}. 
The first step, is the following proposition regarding the expected value and variance of $\Ncrit{\mu}$. 
\begin{prop}\label{prop:gen_exp_var_Fkr_gen}
Let $k\ge1$, $0\leq \mu \leq d$ (excluding $k=1,\mu=0$), and $r\le \rmax$. Then,
\[
\E\{\Ncrit{\mu}\}\approx C_{\mu} n \Lambda^{\mu+k-2}e^{-\Lambda},
\]
where the constant $C_\mu>0$ is given in \eqref{eq:c_mu}.
 If, in addition, we have $\Lambda = \log n + (d+k-2)\log\log n + c(n)$, $c(n) = o(\log\log n)$, then
\[
\mathrm{Var}(\Ncrit{\mu})\approx\E\{\Ncrit{\mu}\}\approx C_{\mu} e^{-c(n)}.\]
Moreover, in this case, for $\mu=d$ we have
\[
|\mathrm{Var}(\Ncrit{d}) - \E\{\Ncrit{d}\}| = O\left((\log\log n)^{d+k}(\log n)^{-\frac{1}{2(d+k)}}\right).
\]    
\end{prop}
Note that this proposition implies that both the expectation and variance undergo a phase transition at
$\Lambda=\log n+(\mu+k-2)\log\log n$.
Namely, 
\[
\lim_{n\rightarrow\infty}\E\{\Ncrit{\mu}\}
=
\lim_{n\rightarrow\infty}\mathrm{Var}(\Ncrit{\mu})
=
\begin{cases}
0 &\log n+(\mu+k-2)\log\log n + w(n), \\
\infty & \log n+(\mu+k-2)\log\log n - w(n),
\end{cases}
\]
where $w(n)\rightarrow \infty$ as $n\rightarrow\infty$.
In the following, we provide the detailed proof for the expectation and convergence rate parts. We postpone the proof for the variance part to Section \ref{sec:var_proof}, due to its length and technical details.
\begin{proof}[Proof of Proposition \ref{prop:gen_exp_var_Fkr_gen} -- first moment]
Let $0\leq\mu\leq d$ and let $\cX\subset\Rd$ be a set of $m$ points, where by \eqref{eq:I_and_mu} and Theorem \ref{thm:crit_pts} $m$ is in the range $\{\ubar{m}_{\mu},\ldots,\bar{m}_{\mu}\}$, with
\begin{equation}\label{eq:m_min_max}
\ubar{m}_{\mu}=\max\{2,\mu+1\},\quad\text{and}\quad \bar{m}_{\mu} = \min\{d+1,\mu+k\}.
\end{equation}

Recall from \eqref{eq:crit_ball} the definitions of $S(\cX)$, $c(\cX)$, $\rho(\cX)$, $\sigma(\cX)$, $\cB(\cX)$, and $\cI(\cX,\cP)$.
Define
\begin{equation}\label{eq:g_r}
\begin{split}
&
\hcrit(\cX) \coloneqq \ind\{c(\cX)\in\sigma(\cX)\},
\\
&
\gcrit^{i}(\cX,\cP_n) \coloneqq \ind\{\cI(\cX,\cP_n)=i\},
\end{split}
\end{equation}
so that $\hcrit(\cX)$ verifies the first condition
of Theorem \ref{def:crit_pnt}. The second condition of Theorem \ref{def:crit_pnt} is verified by limiting the values $i$ can take.
In addition, define
\begin{equation}\label{eq:grk}
\begin{split}
h_r(\cX)&\coloneqq\hcrit(\cX)\ind\{\rho(\cX)\in(r,r_{\max}]\},
\\
g_r^i(\cX,\cP) & \coloneqq \gcrit^{i}(\cX,\cP_n)h_r(\cX).
\end{split}
\end{equation}
Let $\cP_n^{(m)}$ be the collection of all subsets of $\cP_n$ of size $m$.
Using these notations and \eqref{eq:I_and_mu}, we can express the  number of critical points $N_{k,r}^\mu$ \eqref{eqn:N_k} as 
\begin{equation}\label{eq:Fmr_cov}
    \Ncrit{\mu}
    =
    \sum_{m=\ubar{m}_{\mu}}^{\bar{m}_{\mu}}
    \sum_{\cX\in\cP_n^{(m)}}
    g_r^{\mu+k-m}(\cX,\cP_n).
\end{equation}
By taking the expectation and applying Mecke's formula (see Theorem \ref{th:Palm}), we have
\[
\E\{\Ncrit{\mu}\}= 
\sum_{m=\ubar{m}_{\mu}}^{\bar{m}_{\mu}}
\frac{n^{m}}{m!} \E\{g_r^{\mu+k-m}(\cX',\cP_n\cup\cX')\},
\]
where $\cX'$ is a set of $m$ i.i.d.~random variables uniformly distributed  in $\T^d$, independent of $\cP_n$. 
Taking the expected value of $g_{r}^{\mu+k-m}(\cX,\cP)$ conditioned on $\cX'$, we have, by the properties of the Poisson process $\cP_n$,
\begin{equation}\label{eq:g_con_exp}
    \begin{split}
  \E\{g_r^{\mu+k-m}(\cX',\cP_n\cup\cX')\ |\ \cX' = \bx\}
  &
  =
  h_r(\bx)\prob(|B(\bx)\cap\cP_n|=\mu+k-m)
  \\
  &
  =
  h_r(\bx)\frac{\param{n\omega_d\rho(\bx)^d}^{\mu+k-m}}{(\mu+k-m)!} e^{-n\omega_d\rho(\bx)^d}.      
    \end{split}
\end{equation}
This leads to,
\begin{equation*}
    \begin{split}
        \E\{g_r^{\mu+k-m}(\cX',\cP_n\cup\cX')\}
        &
        =
        \frac{(n\omega_d)^{\mu+k-m}}{(\mu+k-m)!}
        \int_{(\T^d)^{m}}h_r(\bx)
        \rho(\bx)^{d(\mu+k-m)} e^{-n\omega_d\rho(\bx)^d}d\bx.
    \end{split}
\end{equation*}
In order to simplify the above integral, we apply the Blaschke-Petkantschin (BP) formula (Lemma \ref{lem:bp}). 
This change of variables replaces the $m$-point configuration $\bx$, with variables that first locate the minimal circumsphere $S(\bx)$ and then use spherical coordinates on  $S(\bx)$. This leads to,
\begin{equation*}
    \begin{split}
        &\E\{g_r^{\mu+k-m}(\cX',\cP_n\cup\cX')\}\\
        &
         \qquad=\frac{(n\omega_d)^{\mu+k-m}}{(\mu+k-m)!}D_{\mathrm{bp}}\int_{0}^{\infty}\int_{(\mathbb{S}^{m-2})^{m}}
\rho^{d(\mu+k-m)}h_r(\rho\bs{\theta})e^{-n\omega_d\rho^d}\rho^{d(m-1)-1}(V_{\mathrm{simp}}(\bs{\theta}))^{d-m+2}d\bs{\theta}d\rho
\\
          &
        \qquad=
        C_{\mu,m} n^{\mu+k-m}
        \int_{r}^{r_{\max}}\rho^{d(\mu+k-1)-1} e^{-n\omega_d\rho^d}d\rho,
    \end{split}
\end{equation*}
where 
\[
C_{\mu,m}
\coloneqq
D_{\mathrm{bp}}\frac{\omega_d^{\mu+k-m}}{(\mu+k-m)!}\int_{(\mathbb{S}^{m-2})^{m}}\hcrit(\bs{\theta})(\vsimp(\bs{\theta}))^{d-m+2}d\bs{\theta},    
\]
and where $\mathbb{S}^{m-2}$ is the unit sphere in $\R^{m-1}$, $\vsimp(\bth)$ is the volume of the simplex spanned by $\bth$, and $D_{\mathrm{bp}}$ is a constant defined in Lemma \ref{lem:bp}.
By Lemma \ref{lem:gamma}, we have
\[
\begin{split}
            \E\{g_r^{\mu+k-m}(\cX',&\cP_n\cup\cX')\}
            \\
            &
            =
            C_{\mu,m} n^{\mu+k-m}\frac{1}{d(n\omega_d)^{\mu+k-1}}\param{\Gamma(\mu+k-1,\Lambda)-\Gamma(\mu+k-1,\Lambda_{\max})}
            \\
            &
            \approx
            \tilde{C}_{\mu,m}
            n^{-m+1}\Lambda^{\mu+k-2}e^{-\Lambda}
            ,
    \end{split}
\]
where $\Gamma(p,t) = \int_t^\infty z^{p-1}e^{-z}dz $ is the \emph{upper incomplete gamma function},  $\Lambda_{\max} := n\omega_d r^d_{\max}$, and 
\begin{equation}\label{eqn:tC_mu_m}
    \tilde{C}_{\mu,m}\coloneqq \frac{C_{\mu,m}}{d\omega_d^{\mu+k-1}}.
\end{equation}
Putting this back into \eqref{eq:Fmr_cov}, we  have 
\begin{equation*}
\E\{N_{i,r}^{\mu}\}
\approx
\sum_{m=\ubar{m}_{\mu}}^{\bar{m}_{\mu}}
\frac{\tilde{C}_{\mu,m}}{m!}
n\Lambda^{\mu+k-2}e^{-\Lambda}
=
C_{\mu}n\Lambda^{\mu+k-2}e^{-\Lambda},    
\end{equation*}
where 
\begin{equation}\label{eq:c_mu}
C_{\mu}\coloneqq\sum_{m=\ubar{m}_{\mu}}^{\bar{m}_{\mu}}\frac{\tilde{C}_{\mu,m}}{m!},
\end{equation}
which completes the proof.
\end{proof}

\begin{proof}[Proof of Theorem \ref{thm:prob_Fkr_cov}]
Let $\Lambda=\log n + (\mu+k-2)\log\log n + w(n)$, then $\E\{\Ncrit{\mu}\}\approx C_{\mu}e^{-w(n)}\rightarrow 0$, as $n\rightarrow\infty$. Hence, by Markov's inequality, we have 
\[
\prob(\Ncrit{\mu}>0)\leq E\{\Ncrit{\mu}\}\rightarrow 0.
\]
For $\Lambda = \log n + (\mu+k-2)\log\log n -w(n)$, we use Chebyshev's inequality, namely,
\[
\prob(\Ncrit{\mu}=0) 
\leq
\prob(|\Ncrit{\mu} - \E\{\Ncrit{\mu}\}| >\E\{\Ncrit{\mu}\})
\leq
\frac{\mathrm{Var}(\Ncrit{\mu})}{\E\{\Ncrit{\mu}\}^2}.
\]
By Lemma \ref{prop:gen_exp_var_Fkr_gen} we have $\E\{\Ncrit{\mu}\}\approx\mathrm{Var}(\Ncrit{\mu})$, and since $\E\{\Ncrit{\mu}\}\approx C_{\mu} e^{w(n)}\rightarrow \infty$, we have $\prob(\Ncrit{\mu}=0)\rightarrow0$.  
This completes the proof.
\end{proof}

\subsection{Phase transition for k-coverage}\label{sec:coverage}
Our analysis of critical points in the previous section, will now pay off, 
providing a very simple proof for Theorem \ref{thm:coverage_Td}.

\begin{proof}[Proof for Theorem \ref{thm:coverage_Td}]
Let $f:\cM\to \R$, and $f^{-1}((-\infty,\alpha])$ be  the sublevel set. Setting $f_{\max}  = \max_{\cM} f$, it is clear that $f^{-1}((-\infty,\alpha]) = \cM$ if and only if $\alpha \ge f_{\max}$. For a Morse function on a closed manifold $\cM$ (no boundary), the maximum is obtained at a local maximum, which is a critical point of index $d$ (assuming $\dim(\cM) = d$). In our setting this observation translates to the fact that the event $\cC_r^{(k)}$ occurs if and only if $\Ncrit{d} = 0$.

The result then immediately follows from 
Theorem \ref{thm:prob_Fkr_cov}. 
\end{proof}

\subsection{Poisson limit for last vacant components}\label{sec:poisson_lim}
In this section we provide the proof for Theorem \ref{thm:Poisson_lim}, namely the distribution of the last vacant components in the critical window (right before $k$-coverage is reached). 

Recall the definition of the $k$-vacancy process $\cV_r^{(k)}$ \eqref{eq:Vacancy}. We first want to establish a one-to-one correspondence between the connected components of $\cV_r^{(k)}$ and the critical points of $d_{\cP_n}^{(k)}$.

\begin{lem}\label{lem:oneone}
    Let $\Lambda = \log n + (d+k-2)\log\log n + \lambda_0$. Denote by $V_{k,r}$ the number of connected components of $\cV_r^{(k)}$. Then
    \[
        \limninf\prob(V_{k,r} = \Ncrit{d}) = 1.
    \]
\end{lem}

\begin{proof}
Let $\bar d_{\cP_n}^{(k)} = -d_{\cP_n}^{(k)}$, and note that $\cV_{r}^{(k)}$ can be viewed as the sublevel sets of $\bar d_{\cP_n}^{(k)}$. From Morse Theory, we know that critical points of index $0$ of $\bar d_{\cP_n}^{(k)}$ generate new connected components in $\cV_r^{(k)}$, and these might merge at critical points of index $1$. In addition, every point of index $\mu$ for $d_{\cP_n}^{(k)}$ is a critical point of index $d-\mu$ for $\bar d_{\cP_n}^{(k)}$ and vice versa. 

When $\Lambda = \log n + (d+k-2)\log\log n +\lambda_0$ we know from Theorem \ref{thm:prob_Fkr_cov} that $N_{k,r}^{d-1} = 0$, implying that $\bar d_{\cP_n}^{(k)}$ does not have any critical points of index $1$ in the range $[r,\rmax]$. Therefore, the components of $\cV_r^{(k)}$ generated in the critical window, cannot merge. This implies  that $V_{k,r}$ must equal the number of critical points of index $0$ of $\bar d_{\cP_n}^{(k)}$, which in turn is equal to $\Ncrit{d}$. This completes the proof. 
\end{proof}

Lemma \ref{lem:oneone} implies that we can reduce the problem of computing the distribution of last vacant components to computing the distribution of the critical points of index $d$ that arise in the critical window. Using the results from \cite{bobrowski2022PoissonProcessApprox}, we will show that in the critical window these critical points follow a Poisson limit in the $(d+1)$ dimensional space of $\mathrm{location}\!\times\!\mathrm{size}$.

Recall that $r_0$ is defined via $n\omega_d r_0^d = \log n + (d+k-2)\log\log n + \lambda_0$, where $\lambda_0\in\R$ is fixed, and set $R_0 = \sqrt{r_0}$.
We are interested in the critical points $c\in\T^d$ of index $d$, with $\rho_c\in(r_0, R_0]$. 

Let $\cX\subset\cP_n$ such that $\cX$ generates a critical point $c=c(\cX)$ of index $d$. Slightly abusing notation, we redefine
\[
\begin{split}
h_r(\cX)& \coloneqq h_{\mathrm{crit}}(\cX)\ind\{\rho(\cX)\in(r,R_0]\},\\
g_{r}^{i}(\cX,\cP_n)&\coloneqq g_{\mathrm{crit}}^{i}(\cX,\cP_n)h_r(\cX),
\end{split}    
\]
where $h_{\mathrm{crit}},g_{\mathrm{crit}}^i$ are defined in \eqref{eq:g_r}.

Recall the definition of process  $\xi_k$ \eqref{eq:xi_k}, as a point process in $\mathbb{Y}\coloneqq \T^d\times \R_0$, with $\R_0=[\lambda_0,\infty)$. Then $\xi_k$ can be written as
\[
    \xi_k =\frac{1}{(d+1)!} \sum_{\cX\in\cP_n^{(d+1)}} 
    g_{r_0}^{k-1}(\cX,\cP_n)\delta_{(c(\cX),\lambda(\cX))},
\]
where $\lambda(\cX) = \lambda_{c(\cX)}$ and $\lambda_c$ is defined in \eqref{eq:lambda_c}.

Recall that $\zeta_k$ is a Poisson process on $\mathbb{Y}$ with intensity $C_d e^{-\lambda}d\lambda dc$. Denote by $\mathbf{L},\mathbf{M}$ the intensity measures of $\xi_k,\zeta_k$, respectively, i.e., for any measurable set $A\subset\mathbb{Y}$, we have $\mathbf{L}(A)=\E\{\xi_k(A)\}$ and $\mathbf{M}(A)=\E\{\zeta_k(A)\}$. In addition, we define
\begin{equation}\label{eq:A_hat}
\hat{A}\coloneqq\{(c,s):(c,\lambda)\in A,\ s=\log n + (d+k-2)\log\log n +\lambda\}.    
\end{equation}
The following lemma will be useful later.

\begin{lem}\label{lem:L}
Let $A\subset\mathbb{Y}$. Then,
\[
\mathbf{L}(A) = \E\{\xi_k(A)\}=C_d n
\int_{\hat{A}}
s^{d+k-2}
e^{-s}
\ind\{s\leq n\omega_d R_0^d\}
d s dc.
\]
\end{lem}
\begin{proof}
Using Mecke's formula and \eqref{eq:g_con_exp}, we have
\begin{equation*}
\begin{split}
\mathbf{L}(A) 
&
= 
\frac{n^{d+1}}{(d+1)!}\int_{\left(\T^d\right)^{d+1}}\ind\{(c(\bx),\lambda(\bx))\in A\}\E\{g_{r_0}^{k-1}(\bx,\cP_n\cup\bx)\ \}
d\bx    
\\
&
=
\frac{n^{d+1}}{(d+1)!}\int_{\left(\T^d\right)^{d+1}}\ind\{(c(\bx),\lambda(\bx))\in A\}
h_{r_0}(\bx)\frac{\param{n\omega_d\rho(\bx)^d}^{k-1}}{(k-1)!} e^{-n\omega_d\rho(\bx)^d}
d\bx. 
\end{split}    
\end{equation*}
Applying the Blaschke-Petkanschin formula (Lemma \ref{lem:bp})
we have
\begin{equation}\label{eq:LA}
\begin{split}
\mathbf{L}(A) 
&
= 
D_{\mathrm{bp}}
n^{d+k}
\int_{\T^d}
\int_{0}^{\infty}
\ind\{(c,\lambda(\rho))\in A\}\ind\{\rho\in(r_0,R_0]\}\rho^{d(d+k-1)-1}e^{-n\omega_d\rho^d}
d\rho dc
\\
&
\times
\int_{\left(\mathbb{S}^{d-1}\right)^{d+1}}
\hcrit(\bs{\theta}) \vsimp(\bs{\theta})
d\bs{\theta},
\end{split}    
\end{equation}
where, abusing notations, we set 
\[
\lambda(\rho)\coloneqq n\omega_d\rho^d-\log n - (d+k-2)\log\log n.
\]
Taking the change of variables $n\omega_d\rho^d\rightarrow s$ we have
\begin{equation*}
\begin{split}
\int_{\T^d}
\int_{r_0}^{R_0}
\ind\{(c,\lambda(\rho))\in A\}
&
\rho^{d(d+k-1)-1} 
e^{-n\omega_d\rho^d}
d\rho dc
\\
&
=
\frac{1}{d(n\omega_d)^{d+k-1}}
\int_{\T^d}
\int_{n\omega_d r_0^d}^{n\omega_d R_0^d}
\ind\{(c,s)\in \hat{A}\}
s^{d+k-2}
e^{-s}
d s dc
\\
&
=
\frac{1}{d(n\omega_d)^{d+k-1}}
\int_{\hat{A}}
s^{d+k-2}
e^{-s}
\ind\{s\leq n\omega_d R_0^d\}
ds dc
\end{split}
\end{equation*}
(the indicator does not specify the lower bound since $A\subset \mathbb{Y}$).
Defining 
\begin{equation}\label{eq:C_d}
C_d\coloneqq \frac{D_{\mathrm{bp}}}{d\omega_d^{d+k-1}}\int_{\left(\mathbb{S}^{d-1}\right)^{d+1}}
\hcrit(\bs{\theta}) \vsimp(\bs{\theta})
d\bs{\theta}, 
\end{equation}
and putting all back into \eqref{eq:LA}
completes the proof.
\end{proof}

\begin{proof}[Proof of Theorem \ref{thm:Poisson_lim}] 
The proof consists of  similar steps to the proof of Theorem 6.2 in \cite{bobrowski2022PoissonProcessApprox}.
Using Theorem 4.1 in \cite{bobrowski2022PoissonProcessApprox} to our setting, we have
\begin{equation}\label{eq:Dkr}
\begin{split}
d_{\mathrm{KR}}(\xi_k,\zeta_k) \leq & d_{\mathrm{TV}}(\mathbf{L},\mathbf{M})
\\
&
+ 2 \{\mathrm{Var}(\xi_k(\mathbb{Y}))-\E(\xi_k(\mathbb{Y}))\}
\\
&
+ \left(\frac{2n^{d+1}}{(d+1)!}\right)^2\int_{(\T^d)^{d+1}}\int_{(\T^d)^{d+1}}
\ind\{\cB(\bx)\cap \cB(\zint)\neq\emptyset\}
\\
&
\times 
\E\{g_{r_0}^{k-1}(\bx,\cP_n\cup \bx)\}\E\{g_{r_0}^{k-1}(\zint,\cP_n\cup\zint)\}
d\zint d\bx,
\end{split}
\end{equation}
where $\cB(\bx)$ denotes the $d$-dimensional ball centered at $c(\bx)$ with radius $\rho(\bx)$ (similarly for $\cB(\zint)$). 
Our goal is to bound the right hand side.

For the first term, recall that $\zeta_k$ is a Poisson process on $\mathbb{Y}$ with intensity $C_d e^{-\lambda}d\lambda d c$. Hence,
\[
\mathbf{M}(A) = \E\{\zeta_k(A)\} = C_d \int_A e^{-\lambda}d\lambda d c.
\]
By Lemma \ref{lem:L} we have
\[
\mathbf{L}(A) = \E\{\xi_k(A)\}=C_d n
\int_{\hat{A}}
s^{d+k-2}
e^{-s}
\ind\{s\leq n\omega_d R_0^d\}
d s dc,
\]
where $\hat{A}$ is defined in \eqref{eq:A_hat}.
Defining $I_n(\lambda)\coloneqq\ind\{\log n+(d+k-2)\log\log n +\lambda\leq n\omega_d R_0^d\}$, we have
\[
\begin{split}
\mathbf{L}(A)
&
=
C_d n
\int_{A}
(\log n+(d+k-2)\log\log n +\lambda)^{d+k-2}
e^{-(\log n+(d+k-2)\log\log n +\lambda)}
I_n(\lambda)
d\lambda dc
\\
&
=
C_d
\int_{A}
\left(1+\frac{(d+k-2)\log\log n +\lambda}{\log n}\right)^{d+k-2}
e^{-\lambda}
I_n(\lambda)
d\lambda dc.
\end{split}    
\]
If $n$ is large enough, then $\lambda  \ge \lambda_0 >-\log\log n$. Hence,
\begin{equation*}
\begin{split}
& 
|\mathbf{L}(A) - \mathbf{M}(A)| 
\\
&
=
C_d
\left|\int_A\left(I_n(\lambda)
\left(1+\frac{(d+k-2)\log\log n +\lambda}{\log n}\right)^{d+k-2} - 1\right)e^{-\lambda}
d\lambda dc
\right|
\\
&
\leq
C_d
\sum_{i=1}^{d+k-2}\binom{d+k-2}{i}\int_A
\left(\frac{(d+k-2)\log\log n +\lambda}{\log n}\right)^{i}e^{-\lambda}
d\lambda dc
+
C_d\int_{\R}\left(1-I_n(\lambda)\right)e^{-\lambda}d\lambda
\\
&
=
C_d
\sum_{i=1}^{d+k-2}\sum_{j=0}^{i}\binom{d+k-2}{i}\binom{i}{j}
\frac{\left((d+k-2)\log\log n\right)^{i-j}}{(\log n)^i}
\int_A \lambda^{j}e^{-\lambda}
d\lambda dc
\\
& 
\quad
+
C_d e^{-\omega_d n R_0^d +\log n + (d+k-2)\log\log n}
\\
&
=
O\left(\frac{\log\log n}{\log n}\right),
\end{split}    
\end{equation*}
since $\int_{\R_0}\lambda_j e^{-j}d\lambda\leq\infty$ and $R_0=\sqrt{r_0}$.

For the second term in \eqref{eq:Dkr}, note that $\xi_k(\mathbb{Y})$ is the number of critical points of index $d$, with $\rho\ge r_0$. Therefore, using Proposition \ref{prop:gen_exp_var_Fkr_gen} we have 
\[
|\mathrm{Var}(\xi_k(\mathbb{Y}))-\E\{\xi_k(\mathbb{Y})\}|
=
O\left((\log\log n)^{d+k}(\log n)^{-\frac{1}{2(d+k)}}\right).
\]
In particular, $\mathrm{Var}(\xi_k(\mathbb{Y}))-\E\{\xi_k(\mathbb{Y})\}\rightarrow 0$ as $n\rightarrow\infty$.

We are left with last term of \eqref{eq:Dkr}. Note that $\cB(\bx)\cap \cB(\zint)\neq\emptyset$ implies $|c(\bx)-c(\zint)|\leq \rho(\bx)-\rho(\zint)\leq 2R_0$. Hence, fixing $\bx$ we have
\begin{equation*}
\begin{split}
\frac{n^{d+1}}{(d+1)!} & \int_{(\T^d)^{d+1}}
\ind\{\cB(\bx)\cap \cB(\zint)\neq\emptyset\}\E\{g^{k-1}_{r_0}(\zint,\cP_n\cup\zint)\}
d\zint
\\
&
\leq
\frac{n^{d+1}}{(d+1)!}
\int_{(\T^d)^{d+1}}
\ind\{c(\zint)\in B_{2R_0}(c(\bx))\}\E\{g_{r_0}^{k-1}(\zint,\cP_n\cup\zint)\}
d\zint
\\
&
= 
\mathbf{L}(B_{2R_0}(c(\bx))\times\R_0) 
=
\vol(B_{2R_0}(c(\bx)))\mathbf{L}(\T^d\times\R_0)
\\
&
=
\omega_d(2R_0)^d C_de^{-\lambda_0}.
\end{split}    
\end{equation*}
Finally,
\begin{equation*}
\begin{split}
\left(\frac{n^{d+1}}{(d+1)!}\right)^2&\int_{(\T^d)^{d+1}}\int_{(\T^d)^{d+1}}
\ind\{\cB(\bx)\cap \cB(\zint)\neq\emptyset\}
\\
&
\times 
\E\{g_{r_0}^{k-1}(\bx,\cP_n\cup\bx)\}\E\{g_{r_0}^{k-1}(\zint,\cP_n\cup\zint)\}
d\zint d\bx
\\
&
\leq
\omega_d(2R_n)^d C_de^{-\lambda_0}
\frac{n^{d+1}}{(d+1)!}
\int_{(\T^d)^{d+1}}
\E\{g_{r_0}^{k-1}(\bx,\cP_n\cup\bx)\}
d\bx
\\
&
= 
O(R_0^d)
.    
\end{split}    
\end{equation*}
Since $R_0=\sqrt{r_0}$, we have $R_0^d=O\left(\sqrt{\log n/n}\right)$. Placing the three bounds we obtained into \eqref{eq:Dkr} completes the proof.
\end{proof}

\subsection{The Euler characteristic}
In this section we prove Proposition \ref{prop:euler_char}.
Note that the following proof holds only for the case $k>1$, while the case $k=1$ was already studied in \cite{bobrowski_vanishing_2017}. There are two main differences between $k=1$ and $k>1$. The first is that 
in the $k=1$  case the critical points of index $\mu=0$ (minima) are exactly $\cP_n$, and they all appear at level zero. For $k>1$ this is no longer true and critical points of index $0$ appear at random levels. The second difference is that for $k=1$, and for any $\mu>1$, each critical point of index $\mu$ introduces exactly one change in the homology. For $k>1$ the number of changes at each critical point is random, which we denoted previously by $\Delta_c$. This makes the analysis of the Euler characteristic more challenging.

Denote $\cC_r^\mu(\cP_n)$ the set of all critical points $c\in \T^d$ of index $\mu$, with critical value $\rho_c\le r$. Recall that the number of changes in homology induced by a critical point $c$ is $\Delta_c$ given by \eqref{eq:delta_c}. 
\begin{lem}\label{lem:euler_formula}
The Euler characteristic of $B_r^{(k)}$ is given by
\[
\chi(B_r^{(k)}) = \sum_{\mu=0}^{d}(-1)^{\mu}\sum_{c\in \cC^{\mu}_r(\cP_n)}\Delta_c.
\]
\end{lem}

\begin{proof}
Note that $B_0^{(k)} = \emptyset$, and consequently $\beta_i(B_0^{(k)}) = 0$ for all $i$.
From Morse Theory, and specifically Theorem 2 in \cite{reani2024knn}, we know that increasing the radius from $0$ to $r$, the Betti numbers will only change at the critical levels.
Each critical point $c$ with $\rho_c \le r$ and index $\mu_c$, increases $\beta_{\mu_c}$ by $\Delta_c^+$ and decreases $\beta_{\mu_c-1}$ by $\Delta_c^-$, where $\Delta_c^+,\Delta_c^-\ge 0$, and $\Delta_c = \Delta_c^+ + \Delta_c^-$. Summing up these changes, we thus have
\[
    \beta_\mu(B_r^{(k)}) = \sum_{c\in \cC_r^\mu(\cP_n)} \Delta_c^+ - \sum_{c\in \cC_r^{\mu+1}(\cP_n)} \Delta_c^-.
\]
Therefore,
\[
\begin{split}
\chi(B_r^{(k)}) &= \sum_{\mu=0}^d (-1)^\mu\beta_\mu(B_r^{(k)}) 
\\ 
&= \sum_{\mu=0}^d \sum_{c\in \cC_r^\mu(\cP_n)} (-1)^\mu\Delta_c^+ + \sum_{\mu=1}^{d}\sum_{c\in \cC_r^\mu(\cP_n)} (-1)^{\mu} \Delta_c^- \\ 
&= \sum_{\mu=0}^d\sum_{c\in \cC_r^\mu(\cP_n)}  (-1)^\mu \Delta_c.
\end{split}
\]
This completes the proof.
\end{proof}

\begin{proof}[Proof of Proposition \ref{prop:euler_char}]
Applying Lemma \ref{lem:euler_formula} and taking the expectation we have
\[
\E\{\chi(B_r^{(k)})\} = \sum_{\mu=0}^{d}(-1)^{\mu}
\E\left\{\sum_{c\in \cC^{\mu}_r(\cP_n)}\Delta_c\right\}.
\]

Fix $0\le \mu \le d$, and let $\cX\subset\cP_n$, $|\cX| = m$, with $\ubar{m}_{\mu}\le m\le\bar{m}_{\mu}$ (see \eqref{eq:m_min_max}), such that  $c=c(\cX)$ is a critical point of index $\mu$. Denote $\Delta_{\mu,m}\coloneqq\Delta_{c}=\binom{m-1}{\mu}$. In addition, recall the definitions of $\hcrit(\cX)$,  and $\gcrit^{i}(\cX,\cP_n)$ in \eqref{eq:g_r}, and denote
\begin{equation*}
\begin{split}
\tilde{h}_r(\cX) & \coloneqq \hcrit(\cX)\ind\{\rho(\cX)\le r\},
\\
\tilde{g}^{i}_r(\cX,\cP_n) &\coloneqq \gcrit^{i}(\cX,\cP_n)\tilde{h}_r(\cX).
\end{split}    
\end{equation*}
Then,
\begin{equation*}
\begin{split}
\E\left\{\sum_{c\in \cC^{\mu}_r(\cP_n)}\Delta_c\right\}
&
= 
\E\left\{\sum_{m=\ubar{m}_{\mu}}^{\bar{m}_{\mu}}\sum_{\cX\in\cP_n^{(m)}} \Delta_{\mu,m} \tilde{g}_r^{\mu+k-m}(\cX,\cP_n)\right\}    
\\
&
=
\sum_{m=\ubar{m}_{\mu}}^{\bar{m}_{\mu}}\Delta_{\mu,m} 
\E\left\{\sum_{\cX\in\cP_n^{(m)}} \tilde{g}_r^{\mu+k-m}(\cX,\cP_n)\right\}.
\end{split}    
\end{equation*}

Following similar calculations as in the proof of Proposition \ref{prop:gen_exp_var_Fkr_gen}, we  have
\begin{equation*}
\begin{split}
\E\left\{\chi(B_r^{(k)})\right\}
&
=
n\sum_{\mu=0}^{d}(-1)^{\mu}\gamma(\mu+k-1,\Lambda)
\sum_{m=\ubar{m}_{\mu}}^{\bar{m}_{\mu}}\tilde{C}_{\mu,m}\frac{\Delta_{\mu,m}}{m!}
\\
&
=
n\sum_{\mu=0}^{d}(-1)^{\mu}\tilde{C}_{\mu}\gamma(\mu+k-1,\Lambda),
\end{split}
\end{equation*}
where $\tilde C_{\mu,m}$ was defined in \eqref{eqn:tC_mu_m}, $\tilde{C}_{\mu}\coloneqq
\sum_{m=\ubar{m}_{\mu}}^{\bar{m}_{\mu}}\tilde{C}_{\mu,m}\frac{\Delta_{\mu,m}}{m!}$, and $\gamma(p,t) = \int_0^t z^{p-1}e^{-z}dz $ is 
the \emph{lower incomplete gamma function}.
Denoting $D_{\mu} \coloneqq \tilde{C}_{\mu}(\mu+k-2)!$, and using the fact that $\gamma(p,t) = (p-1)!(1-e^{-t}\sum_{i=0}^{p-1} t^i / i!)$ (for an integer $p$), 
we have 
\begin{equation*}
\begin{split}
\E\left\{\chi(B_r^{(k)})\right\}
&
=
n\sum_{\mu=0}^{d}(-1)^{\mu}D_{\mu}\left(1-e^{-\Lambda}\sum_{i=0}^{\mu+k-2}\frac{\Lambda^{i}}{i!}\right)
\\
&
=
n\sum_{\mu=0}^{d}(-1)^{\mu}D_{\mu}-ne^{-\Lambda}\sum_{\mu=0}^{d}(-1)^{\mu}D_{\mu}\sum_{i=0}^{\mu+k-2}\frac{\Lambda^{i}}{i!}
\\
&
=
n\sum_{\mu=0}^{d}(-1)^{\mu}D_{\mu}-ne^{-\Lambda}
\left[
\sum_{\mu=0}^{d}(-1)^{\mu}D_{\mu}\sum_{i=0}^{k-2}\frac{\Lambda^{i}}{i!}
+
\sum_{\mu=0}^{d}(-1)^{\mu}D_{\mu}\sum_{i=k-1}^{\mu+k-2}\frac{\Lambda^{i}}{i!}
\right]
\\
&
=
n\sum_{\mu=0}^{d}(-1)^{\mu}D_{\mu}
-
ne^{-\Lambda}
\left[
\sum_{i=0}^{k-2}\sum_{\mu=0}^{d}(-1)^{\mu}D_{\mu}\frac{\Lambda^i}{i!}
+
\sum_{i=k-1}^{d+k-2}\sum_{\mu=i-k+2}^{d}(-1)^{\mu}D_{\mu}\frac{\Lambda^{i}}{i!}
\right].
\end{split}
\end{equation*}
Setting
\begin{equation}\label{eq:euler_consts}
A_i \coloneqq 
\begin{cases}
\frac{1}{i!}\sum_{\mu=0}^{d}(-1)^{\mu-1}D_{\mu} & 0\leq i\leq k-2, \\ 
\frac{1}{i!}\sum_{\mu=i-k+2}^{d}(-1)^{\mu-1}D_{\mu} & k-1\leq i\leq d+k-2,  
\end{cases}    
\end{equation}
we have
\[
\E\{\chi(B_r^{(k)})\}
=
-nA_0
+
ne^{-\Lambda}
\sum_{i=0}^{d+k-2}A_i\Lambda^i.
\]
Next, fixing $r>0$ and taking $n\rightarrow\infty$ ($\Lambda$ much larger than the coverage threshold), we have
\[
\lim_{n\rightarrow\infty}\E\left\{\chi(B_r^{(k)})\right\}
=
\chi(\T^d)=0.
\]
This implies that $A_0=0$.
\end{proof}

\subsection{Homological connectivity for the $k$-fold cover}
Recall the definition of $\cH_{i,r}^{(k)}$  \eqref{eq:hom_con_event} -- the event where the $i$-th homology of $B_r^{(k)}$ ``stabilizes'' (i.e., no further changes occur for increasing values of $r$), and becomes isomorphic to $H_i(\T^d)$. In this section we prove Corollary \ref{cor:hom_con} regarding the emergence of $\cH_{i,r}^{(k)}$.
 
\begin{proof}[Proof of Corollary \ref{cor:hom_con}]
 We start with the cases $i=d-1,d$.
Above the threshold, namely, $\Lambda=\log n + (d+k-2)\log\log n + w(n)$, we have  from Theorem \ref{thm:coverage_Td} that $\T^d$ is $k$-covered with high probability.  In other words, $B_r^{(k)} = \T^d$, and therefore, both $\cH_{d-1,r}^{(k)},\cH_{d,r}^{(k)}$ occur. Below the threshold, 
when $\Lambda=\log n + (d+k-2)\log\log n - w(n)$, we have that $\T^d$ is not $k$-covered with high probability,  
and therefore, $\cH_{d,r}^{(k)}$ does not hold. To show that $\cH_{d-1,r}^{(k)}$ does not hold as well, we use the following insight.

Note that from \eqref{eq:I_and_mu}, for $\mu=d$, since $\cI(\cX,\cP)\le k-1$, and $|\cX|\le d+1$, we must have $|\cX| = d+1$ and $\cI(\cX,\cP) = k-1$.
Using \eqref{eq:delta_c}, we have that $\Delta_c = \binom{d}{\mu} = 1$. 
In other words, a critical point of index $d$ induces a single change in the homology. This change can either be positive (generating a $d$-cycle) or negative (eliminating a $(d-1)$-cycle). Since the $k$-coverage process $B_r^{(k)}$ is a subset of $\T^d$, only one $d$-cycle can be generated, and this occurs when we reach $k$-coverage. This implies that all critical points of index $d$, except the last one (the global maximum), must eliminate a $(d-1)$-cycle. Therefore, if $N^{k}_{d,r}>1$, we can conclude that there is still a $(d-1)$-cycle waiting to be eliminated, implying that $\cH_{d-1,r}^{(k)}$ has not occurred yet.

Similarly to the proof of Theorem \ref{thm:prob_Fkr_cov} we have
\[
\prob(\Ncrit{d}\le 1) \leq \prob(|\Ncrit{d} - \E\{\Ncrit{d} \}| > \E\{\Ncrit{d} \} - 1)\leq \frac{\mathrm{Var}(\Ncrit{d})}{(\E\{\Ncrit{d}\} -  1)^2} 
\approx
\frac{\mathrm{Var}(\Ncrit{d})}{\E\{\Ncrit{d}\}^2},
\]
and by Lemma \ref{prop:gen_exp_var_Fkr_gen} the last term goes to $0$ as $n\rightarrow\infty$,  completing the proof.

For $0\le i \le d-2$, we note that if  $\Ncrit{i}=\Ncrit{i+1} = 0$, then no more changes can occur in $H_i(B_r^{(k)})$, implying that $\cH_{i,r}^{(k)}$ holds. From Theorem \ref{thm:prob_Fkr_cov} we have that when $\Lambda = \log n + (i+k-1)\log\log n + w(n)$, then $\prob({\Ncrit{i}= \Ncrit{i+1} =0}) \to 1$, completing the proof.
\end{proof}

\subsection{The limiting variance of \headermath{\Ncrit{d}}}\label{sec:var_proof}

The last thing we are left to prove is the asymptotic result for  $\mathrm{Var}(N^{k}_{\mu,r})$ in Proposition \ref{prop:gen_exp_var_Fkr_gen}, namely, that 
$\mathrm{Var}(\Ncrit{\mu})\approx \E(\Ncrit{\mu})$, as well as the convergence rate for $\mu=d$.
For simplicity, we show this only for $\mu=d$. All other cases $0\leq\mu<d$ follow similar calculations, but require more special cases which we review in the end. 

\begin{remark}
While the main steps are similar to \cite{bobrowski2022homological}, the $k$-NN case contains significant challenges.
A key difference between \cite{bobrowski2022homological} and here, is that the calculations in \cite{bobrowski2022homological} use the fact that for any two critical configuration $\cX_1,\cX_2\subset\cP_n$, the open ball $\cB(\cX_1)$ does not intersect $\cX_2$, and $\cB(\cX_2)$ does not intersect $\cX_1$, since for $k=1$ we must have $\cI(\cX_1,\cP_n)=\cI(\cX_2,\cP_n)=0$. This fact significantly limits the possible configurations and simplifies the calculations. However, for $k>1$, this fact no longer holds, since $\cI(\cX_1,\cP_n),\cI(\cX_2,\cP_n)\geq 0$, posing significant new challenges to the moment estimation.  
\end{remark}

\begin{proof}[Proof of Proposition \ref{prop:gen_exp_var_Fkr_gen} -- second moment]
To simplify the notation we denote $N:=\Ncrit{d}$, and we use the symbol $C^*$ to represents any constant value that does not depend on $n$, and its value can change throughout the proof.

Recall the definitions of $h_r(\cX)$, and $g_r^{i}(\cX,\cP_n)$ from \eqref{eq:grk}. In addition, note that by \eqref{eq:I_and_mu} a critical configuration $\cX\subset\cP_n$, inducing an index $\mu=d$ critical point, must have $|\cX|=d+1$ and $\cI(\cX,\cP_n)=k-1$. 
Therefore,
\begin{equation}\label{eq:second_moment}
    \begin{split}
        \E\curparam{N^2}
        & = 
        \E\curparam{\sum_{\cX_1,\cX_2\in\cP_{n}^{(d+1)}} g_{r}^{k-1}(\cX_1, \cP_n)g_{r}^{k-1}(\cX_2, \cP_n)}
        \\
        &
        = 
        \sum_{j=0}^{d+1}\E\curparam{\sum_{\substack{\cX_1,\cX_2\in\cP_n^{(d+1)} \\ 
        |\cX_1\cap\cX_2|=j}} g_{r}^{k-1}(\cX_1, \cP_n)g_{r}^{k-1}(\cX_2, \cP_n)}
        \\
        &
        =
        \sum_{j=0}^{d+1}I_j,
    \end{split}
\end{equation}
where $\cP_{n}^{(d+1)}$ denotes all the subsets of $\cP_n$ of size $(d+1)$, 
and $I_j$ denotes the $j$-th term in the sum. Note that $I_{d+1}=\E\curparam{N}$, and therefore,
\begin{equation}\label{eq:var_explicit}
\mathrm{Var}(N) = 
\E\{N^2\} - \E\{N\}^2
=
\E\{N\} + \sum_{j=1}^{d}I_j + (I_0-\E\{N\}^2).
\end{equation}

Next, let $1\leq j\leq d$. By Definition \ref{def:crit_pnt}, each ball $\cB(\cX_1)$, $\cB(\cX_2)$ contains exactly $(k-1)$ points in its interior. 
Using Mecke's formula (Theorem \ref{cor:palm}), we have
\begin{equation}\label{eq:I_j_palm}
I_j=     
\frac{n^{2(d+1)-j}}{j!((d+1-j)!)^2}
\E\curparam{g_{r}^{k-1}(\cX'_1, \cP_n\cup\cX')g_{r}^{k-1}(\cX'_2, \cP_n\cup\cX')},
\end{equation}
where $\cX'=\cX_1'\cup\cX_2'$ is a set of $(2(d+1)-j)$ i.i.d.~random variables uniformly distributed in $\T^d$, independent of $\cP_n$.  
Taking the expected value conditioned on $\cX'$, we get
\begin{equation*}  
\begin{split}
\E\{g_{r}^{k-1}&(\cX_1, \cP_n\cup\cX')  g_{r}^{k-1}(\cX_2, \cP_n\cup\cX') | \cX' =\bx \} 
\\
&
=
h_r(\bx_1)h_r(\bx_2)\prob\left(\cI(\bx_1,\cP_n\cup\bx)=k-1,\cI(\bx_2,\cP_n\cup\bx)=k-1\right)
,
\end{split}
\end{equation*}
where 
\begin{equation}\label{eq:x1x2x12}
\begin{split}
\bx & = (x_1,\ldots,x_{2(d+1)-j})\subset(\T^d)^{2(d+1)-j},
\\
\bx_1 & = (x_1,\ldots,x_{d+1})
\\
\bx_2 & = (x_1,\ldots,x_j,x_{d+2},\ldots,x_{2(d+1)-j}),    
\end{split}   
\end{equation} 
i.e., $\bx_1$ and $\bx_2$ share the first $j$ coordinates.
We denote 
\[
p(\bx) \coloneqq 
\prob\left(\cI(\bx_1,\cP_n\cup\bx)=k-1,\cI(\bx_2,\cP_n
\cup\bx)=k-1 \right),
\]
and therefore, we have
\[
    I_j =    
    C^* n^{2(d+1)-j}
    \int_{\cA_j}
    p(\bx)
    d\bx,
\]
where $\cA_j = \curparam{\bx\in (\T^{d})^{2(d+1)-j} :  h_r(\bx_1)h_r(\bx_2) = 1}$. By the symmetry between $\bx_1,\bx_2$, we have
\[
\int_{\cA_j}
    p(\bx)
    d\bx
    =
    2\int_{\cA_j'}
    p(\bx)
    d\bx
    ,
\]
where $\cA_j'=\{\bx\in\cA_j:\rho(\bx_1)\geq\rho(\bx_2)\}$.

Next, we write $p(\bx)$ explicitly.
Denote the following volumes (see Figure \ref{fig:notations}), 
\begin{equation*}
V_1(\bx) \coloneqq \vol(\cB(\bx_1){\setminus}\cB(\bx_2)),
\quad
V_2(\bx) \coloneqq \vol(\cB(\bx_2){\setminus}\cB(\bx_1)),
\quad
V_{12}(\bx) \coloneqq \vol(\cB(\bx_1){\cap}\cB(\bx_2)).
\end{equation*}
In addition, we use the following notations to count the points lying in each ball $\cB(\bx_1),\cB(\bx_2)$ (see Figure \ref{fig:notations}),
\begin{equation}\label{eq:m1_m2}
m_1 = |\cB(\bx_1)\cap\bx_2|,
\quad 
m_2 = |\cB(\bx_2)\cap\bx_1|,
\end{equation}
and
\begin{equation}\label{eq:k1_k2_k12}
k_1=|\cP_n\cap \cB(\bx_1)\setminus\cB(\bx_2)|,\quad
k_2=|\cP_n\cap \cB(\bx_2)\setminus\cB(\bx_1)|,\quad
k_{12} = |\cP_n\cap\cB(\bx_1)\cap\cB(\bx_2)|.    
\end{equation}

The values above satisfy the following,
\begin{equation}\label{eq:ks}
m_1 + k_1+ k_{12} = m_2 + k_2+k_{12}=k-1.   
\end{equation}
Note that $0\leq k_i\leq K_{i}\coloneqq k-1-m_i$, for $i\in\{1,2\}$, and $0\leq k_{12}\leq K_{12}\coloneqq k-1-\max_{i\in\{1,2\}}(m_i+k_i)$.

\begin{figure}
    \centering
    \includegraphics[width=0.5\linewidth]{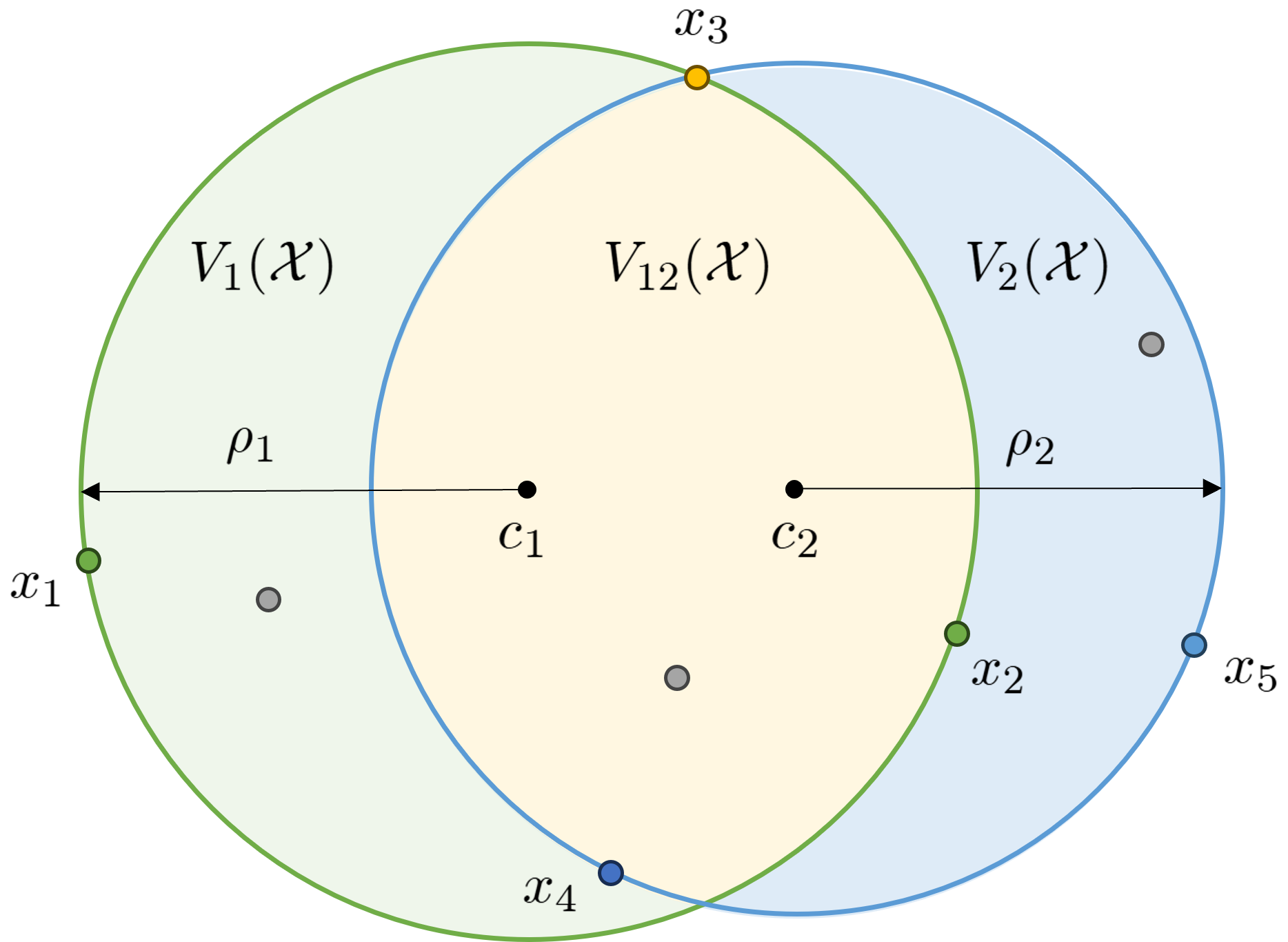}
    \caption{A configuration of two critical points of index $\mu=2$ in $\R^2$ for $k=4$. The set $\cX=\{x_1,\ldots,x_5\}$ induces two critical points, generated by $\cX_1=\{x_1,x_2,x_3\}$ and $\cX_2=\{x_3,x_4,x_5\}$. Here $k_1=1$, $k_2=1$, $k_{12}=1$, $m_1=1$ and $m_2=1$. In addition, $j=1$ since $\cX_1$ and $\cX_2$ share the point $x_3$. The volume $V_1(\cX)$ is green, $V_2(\cX)$ is blue, and $V_{12}(\cX)$ is yellow.}
    \label{fig:notations}
\end{figure}

Using the definitions above, we can write
\[
p(\bx) 
= 
\sum_{k_1=0}^{K_1}\sum_{k_2=0}^{K_2}
\sum_{k_{12}=0}^{K_{12}}
n^{k_1+k_2+k_{12}}
\frac{V_1(\bx)^{k_1}V_2(\bx)^{k_2}V_{12}(\bx)^{k_{12}}}{k_1!k_2!k_{12}!}e^{-n\vuni(\bx_1,\bx_2)}.    
\]

Next, define $\bx_{12}=(x_1,\ldots,x_j)$ and
\begin{equation}
\begin{split}
c_i & = c(\bx_1),\quad \rho_i=\rho(\bx_i),\quad i=1,2,
\\
c_{12} & = c(\bx_{12}),\quad \rho_{12} = \rho(\bx_{12}).
\end{split}    
\end{equation}
We split the integral into three regions, so that
\begin{equation}\label{eq:I_j}
I_j = C^* n^{2(d+1)-j}
    (I_j^{(1)} + I_j^{(2)} + I_j^{(3)}),
\end{equation}
where
\begin{equation}\label{eq:I_j_split}
\begin{split}
    I_j^{(1)} & \coloneqq \int_{\cA_j'} 
    p(\bx)\ind\{|c_1-c_{12}|\leq\epsilon_j\rho_1\} 
    d\bx
    , 
    \\
    I_j^{(2)} & \coloneqq \int_{\cA_j'} 
    p(\bx)\ind\{|c_1-c_{12}|>\epsilon_j\rho_1, |c_1-c_2|\leq\delta_j\rho_1\} 
    d\bx,
    \\
    I_j^{(3)} & \coloneqq \int_{\cA_j'} 
    p(\bx)\ind\{|c_1-c_{12}|>\epsilon_j\rho_1, |c_1-c_2|>\delta_j\rho_1\}
    d\bx,
\end{split}
\end{equation}
where $\epsilon_j$ and $\delta_j$ are chosen later such that both goes to zero as $n\rightarrow\infty$, and $\delta_j=o(\epsilon_j)$.

\vspace{5pt}
\noindent\textbf{Bounding $I_j^{(1)}$.}\label{par:I_j^(1)}
To evaluate $I_j^{(1)}$, we use the partial BP formula for $\bx_2{\setminus}\bx_1$ (Lemma \ref{lem:bp_partial}).
\begin{equation*}
    \begin{split}
        I_j^{(1)} 
        &
        =
        \sum_{k_{1},k_{2},k_{12}}
        \frac{d!n^{k_1+k_2+k_{12}}}{k_1!k_2!k_{12}!(j-1)!}
        \int_{\bx_1} d\bx_1
        \int_{c_2} dc_2
        \int_{\Pi_0^{\perp}} d\Pi_0^{\perp}
        \int_{\hat{\bs{\theta}}_2} d\hat{\bs{\theta}}_2
        \\
        &
        \times 
        h_r(\bx_1)h_r(\bx_2)
        \ind\curparam{\rho_2\leq\rho_1, |c_1 - c_{12}|\leq \epsilon_j\rho_1}
        \\
        &
        \times
        \rho_2^{(d-1)(d+1-j)}\left(\frac{\vsimp(\bs{\theta}_2)}{\vsimp(\bs{\theta}_{12})}\right)
        V_1(\bx)^{k_1}V_2(\bx)^{k_2}V_{12}(\bx)^{k_{12}} e^{-nV_{\mathrm{uni}}(\bx_1, \bx_2)},
    \end{split}
\end{equation*}
where  $\bs{\theta}_2,\bs{\theta}_{12}$ are the spherical coordinates for $\bx_2,\bx_{12}$ \eqref{eq:x1x2x12}, respectively, $\hat{\bs{\theta}}_2=\bs{\theta}_2{\setminus}\bs{\theta}_{12}$, and $V_{\mathrm{uni}}(\bx_1, \bx_2)=\vol(B(\bx_1){\cup}B(\bx_2))$. 
Since $\rho_2\leq \rho_1$, we have  $|c_2-c_{12}|\leq|c_1-c_{12}|\leq \epsilon_j\rho_1.$ 
Hence, since the integrand does not depend on $c_2$, the integral over $c_2$ is merely the volume of a ball of radius $\epsilon_j\rho_1$ 
in the plane $\Pi^{\perp}\cong\R^{d+1-j}$, namely, $\omega_{d+1-j}(\epsilon_j\rho_1)^{d+1-j}$.
Moreover, by the triangle inequality we have $|c_1-c_2|\leq |c_1-c_{12}| + |c_2-c_{12}|$, which implies 
$|c_1-c_2|\leq 2\epsilon_j\rho_1$.
Therefore, $\rho_2 - (\rho_1-|c_1-c_2|)\leq 2\epsilon_j\rho_1$.
Thus, we can bound $V_2(\bx)$ by
\[
V_2(\bx)\leq \frac{\kappa_d}{2}\rho_1^{d-1}\cdot\epsilon_j\rho_1 = \epsilon_j \kappa_d\rho_1^{d}.
\]
where $\kappa_{d}$ is the surface area of the $d$-dimensional unit ball.
In addition, we bound $V_1(\bx),V_{12}(\bx)$ from above by the volume of a $d$-ball, $\omega_d\rho_1^{d}$, and we apply the bounds $\rho_2\leq\rho_1$, $\vsimp(\bs{\theta}_2)/\vsimp(\bs{\theta}_{12})\leq 1$. 
This yields,
\begin{equation*}
    \begin{split}
        I_j^{(1)} 
        &
        \leq
        \sum_{k_{1},k_{2},k_{12}}
        C^* \epsilon_j^{d+1-j+k_2}n^{k_1+k_2+k_{12}}
        \int_{\bx_1} d\bx_1
        \ind\curparam{
        |c_1 - c_{12}|\leq \epsilon_j\rho_1}
        \\
        &
        \times
        h_r(\bx_1)\rho_1^{d(d+1-j)+d(k_1+k_2+k_{12})}
        e^{-n\omega_d\rho_1^d}.
    \end{split}
\end{equation*}

Next, let $\phi(\bx_1)$ denote the normalized distance between $c(\bx_1)$ to the center of the closest $(d-1)$-face of $\sigma(\bx_1)$, namely,
\[
\phi(\bx) = \frac{1}{\rho_1}\min_{\hat{\bx}_i}|c_1 -c(\hat{\bx}_i)|,
\]
where $\hat{\bx}_i\coloneqq(x_1,\ldots,x_{i-1},x_{i+1},\ldots,x_{d+1})\subset(\T^{d})^{d}$, and $c(\hat{\bx}_i)$ is the center of the minimal $(d-1)$-sphere containing $\hat{\bx}_i$. Note that $c(\hat{\bx}_i)$ is included in the $(d-1)$-plane that contains $\sigma(\hat{\bx}_i)$.
Having $|c_1-c_{12}|\leq \epsilon_j\rho_1$ implies $\phi(\bx_1)\leq \epsilon_j$. Hence,
\begin{equation*}
    \begin{split}
        I_j^{(1)}
        &
        \leq
        \sum_{k_{1},k_{2},k_{12}}C^*\epsilon_j^{d+1-j+k_2}
         n^{k_1+k_2+k_{12}}
        \int_{\bx_1} d\bx_1 \ind\{\phi(\bx_1)\leq\epsilon_j\}\\
        &\times h_r(\bx_1) \rho_1^{d(d+1-j+k_1+k_2+k_{12})}
        e^{-n\omega_d\rho_1^d}.
    \end{split}
\end{equation*}

Next, we switch to spherical coordinate, namely, $\bx_1\rightarrow c_1+\rho_1\bs{\theta}_1$. We denote $\phi(\bs{\theta}_1)\coloneqq\phi(\bx_1)$, since by definition $\phi(\bx_1)$ depends only on $\bs{\theta}_1$. Hence, using the BP formulae for $\bx_1$,
\begin{equation*}
    \begin{split}
        I_j^{(1)} 
        &
        \leq
        \sum_{k_{1},k_{2},k_{12}}
        C^*\epsilon_j^{d+1-j+k_2}
        n^{k_1+k_2+k_{12}}
        \int_{r}^{\rmax} \rho_1^{d^2-1} 
        \rho_1^{d(d+1-j + k_1+k_2+k_{12})}
        e^{-n\omega_d\rho_1^d} d\rho_1
        \\
        &
        \times
        \int_{(\mathbb{S}^{d-1})^{d+1}}\ind\{\phi(\bs{\theta}_1)\leq\epsilon_j\}
        d\bs{\theta}_1.
\end{split}
\end{equation*}
By Corollary 5.5 in \cite{bobrowski2022homological}, the last integral is Lipschitz in $(0,\epsilon_j]$, and therefore, it is $O(\epsilon_j)$. For the first integral we apply Lemma \ref{lem:gamma}. Hence,
\begin{equation*}
\begin{split}   
        I_j^{(1)} 
        &
        \leq
        \sum_{k_{1},k_{2},k_{12}}
        C^*\epsilon_j^{d+2-j+k_2}
        n^{-(2d + 1 - j)}\Lambda^{2d -j + k_1+k_2+k_{12}}e^{-\Lambda}
        \\
        &
        =
           C^*\epsilon_j^{d+2-j}
        n^{-(2d + 1 - j)}\Lambda^{2d -j}e^{-\Lambda}
        \sum_{k_{1},k_{2},k_{12}}
     (\epsilon_j\Lambda)^{k_2}\Lambda^{k_1+k_{12}}
        .
    \end{split}
\end{equation*}
Since later we choose $\epsilon_j$ such that $\epsilon_j\Lambda\rightarrow\infty$, we bound the last term from above by setting $k_{1}=k_{2}=k-1$ (implying $k_{12}=0$), and therefore,
\begin{equation}\label{eq:I_j^(1)}
I_j^{(1)} \leq         
C^*\epsilon_j^{d+1-j+k}
n^{-(2d + 1 - j)}\Lambda^{2d -j + 2k-2}e^{-\Lambda}.    
\end{equation}

\noindent\textbf{Bounding $I_j^{(2)}$.}
Similarly to the previous case, since $\rho_1-\rho_2\leq |c_1-c_2|\leq \delta_j\rho_1$, we can bound $V_2(\bx)$ from above by 
$V_2(\bx)\leq\frac{\kappa_d}{2}\delta_j\rho_1^d.$
Hence, similarly to the steps taken above, we have
\begin{equation*}
    \begin{split}
        I_j^{(2)} 
        &
        \leq
        \sum_{k_{1},k_{2},k_{12}}
        C^* \delta_j^{k_2}n^{k_1+k_2+k_{12}}
        \int_{\bx_1} d\bx_1
        h_r(\bx_1)\rho_1^{(d-1)(d+1-j)+d(k_1+k_2+k_{12})} e^{-n\omega_d\rho_1^d}
        \\
        &
        \times 
        \int_{c_2} dc_2
        \ind\{|c_1-c_2|\leq\delta_j\rho_1\}
        ,
    \end{split}
\end{equation*}
The last integral is merely the volume of a ball of radius $\delta_j\rho_1$ 
in the plane $\Pi^{\perp}\cong\R^{d+1-j}$.
Therefore,
\begin{equation}\label{eq:I_j^(2)}
    \begin{split}
        I_j^{(2)} 
        &
        \leq
        \sum_{k_1,k_2,k_{12}}
        C^*\delta_j^{d+1-j+k_2}n^{k_1+k_2+k_{12}}
        \int_{\bx_1} d\bx_1
        h_r(\bx_1)
        \rho_1^{d(d+1-j+k_1+k_2+k_{12})}
        e^{-n\omega_d\rho_1^d}.
    \end{split}
\end{equation}
To evaluate the last upper bound, we split it into two subcases depending on  $k_1$.

\vspace{5pt}
\noindent\underline{The case $k_1<k-1-k_{12}$:}
By applying BP formula, we have
\begin{equation*}
    \begin{split}
        I_j^{(2)} 
        &
        \leq
        \sum_{k_{1},k_{2},k_{12}}
        C^*\delta_j^{d+1-j+k_2}
        n^{k_1+k_2+k_{12}}
        \int_{r}^{\rmax}
        \rho_1^{d(2d+1-j+k_1+k_2+k_{12})-1} e^{-n\omega_d\rho_1^d}
        d\rho_1.
\end{split}
\end{equation*}
Similarly to the previous case, we apply Lemma \ref{lem:gamma}, and rearrange the terms. Hence,
\begin{equation*}
\begin{split}
        I_j^{(2)} 
        &
        \leq
        \delta_j^{d+1-j}n^{-(2d+1-j)}\Lambda^{2d-j}e^{-\Lambda}
        \sum_{k_{1},k_{2},k_{12}}
        C^*
        (\delta_j\Lambda)^{k_2}
        \Lambda^{k_1+k_{12}}.
\end{split}
\end{equation*}
Since later we choose the value of $\delta_j$ such that $\delta_j\Lambda\rightarrow\infty$, we have
\begin{equation}\label{eq:I_j^(2)-case1}
\begin{split}
        I_j^{(2)}
        &
        \leq
        C^*\delta_j^{d-j+k}n^{-(2d+1-j)}\Lambda^{2d-j+2k-3}e^{-\Lambda}
        ,
    \end{split}
\end{equation}
where in the last inequality we replaced the sum with its dominant term, obtained by taking $k_{1}=k-2,k_2=k-1$. 

\vspace{5pt}
\noindent\underline{The case $k_1=k-1-k_{12}$:} 
If $j=d$, then the centers $c_1,c_2,c_{12}$, all lie on the same line. Since later we take $\epsilon_j\gg \delta_j$, we have that $c_1,c_2$ lie on the same side of the $(d-1)$-dimensional hyperplane $\Pi(\bx_{12})$, which implies that $\cB(\bx_1)$ contains more than half the sphere $S(\bx_2)$ (see Figure \ref{fig:I2_j_eq_d}). Recall that $c_2$ being critical implies that $c_2\in\sigma(\bx_2)$, and therefore, the single point in $\bx_2{\setminus}\bx_{12}$ must lie inside $\cB(\bx_1)$, implying $m_1=1$ \eqref{eq:m1_m2}. However, this leads to a contradiction, since the assumption $k_1=k-1-k_{12}$ \eqref{eq:k1_k2_k12} implies $m_1=0$. Therefore, we proceed under the assumption that $0\leq j\leq d-1$.

\begin{figure}
    \centering
    \includegraphics[width=0.45\linewidth]{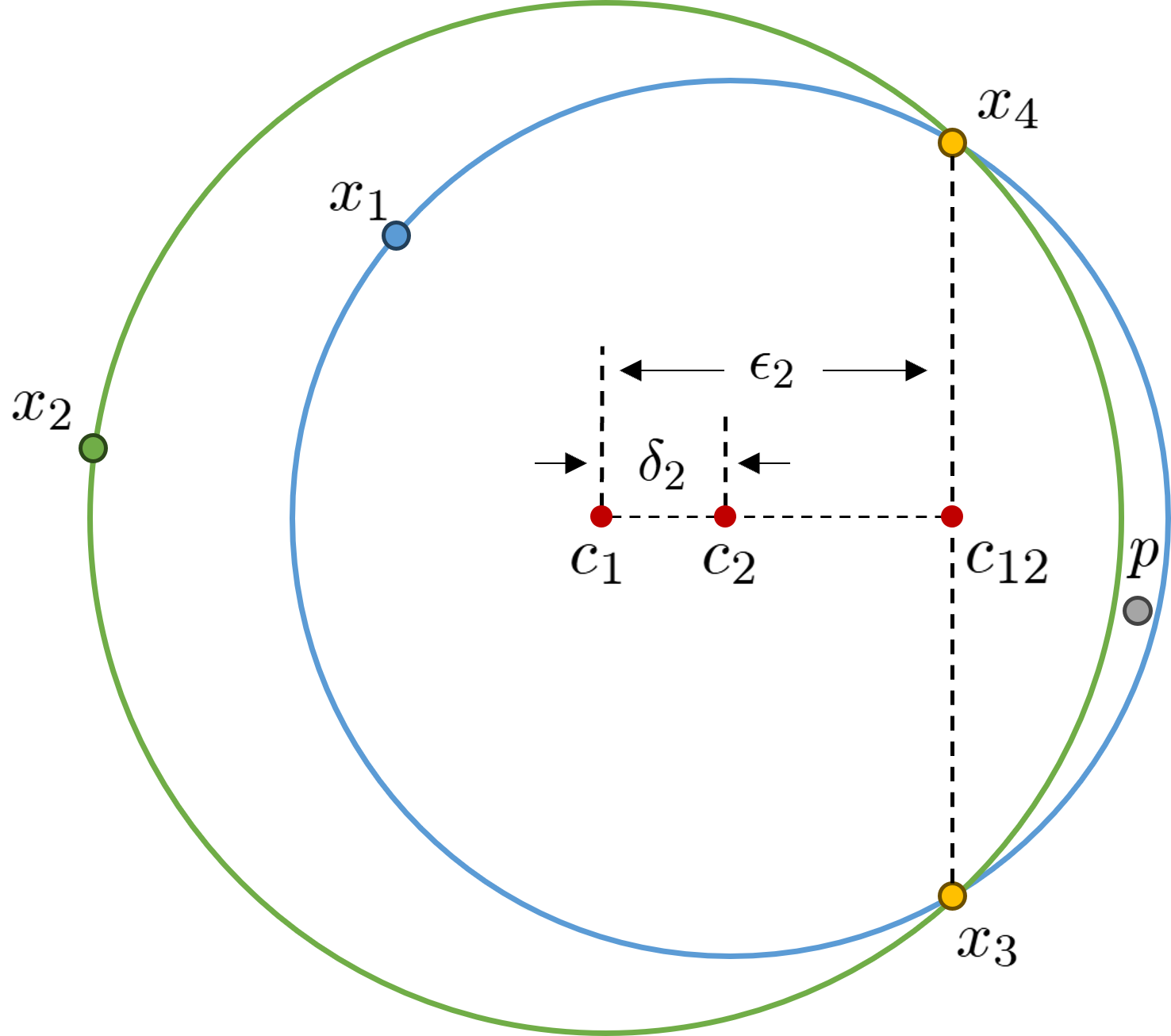}
    \caption{$I_j^{(2)}$ configurations  in $\R^2$ for $k=2$ and $j=2$. The critical points $c_1$ and $c_2$ of index $\mu=2$ are induced by $\cX_1=\{x_1,x_3,x_4\}$ and $\cX_2=\{x_2,x_3,x_4\}$, respectively. Since $j=|\bx_1\cap\bx_2|=2$ and $\delta_j<\epsilon_j$, the centers 
    $c_1$ and $c_2$ lie on the same side of the dashed line connecting $x_3$ and $x_4$, resulting in $\cB(\bx_1)$ (enclosed by the green circle) including more than half the sphere $S(\bx_2)$ (blue circle). Thus, $c_2$ critical implies that at least one point of  $\bx_2$ must lie inside $\cB(\bx_1)$ (the point $x_1$).}
    \label{fig:I2_j_eq_d}
\end{figure}

Note that if $k_1=k-1-k_{12}$, then $m_1 = 0$, so that all points of $\bx_2$ are outside (or on the boundary of) $\cB(\bx_1)$.
For $c_2$ to be critical, we therefore need at least half of the sphere $S(\bx_2)$ to be outside $\cB(\bx_1)$, implying that 
$\rho_2^2\geq \rho_1^2-\Delta^2$.

Take $c_2$ as the origin. Let $P_1$ be the $j$-dimensional subspace that contains $c_2$ and $\bx_{12}$. Let $P_2$ be the $(d-1)$-subspace that is orthogonal to the line connecting $c_1$ and $c_2$. Let $P_3=P_1^{\perp}\cap P_2$, and let $P_4 = P_1\oplus P_3$. The subspace $P_4$ is a $(d-1)$-dimensional subspace that contains $c_2,c_{12}$ and $\bx_{12}$. Let $S_3=P_4\cap S(\bx_2)$, which is a $(d-2)$-dimensional sphere in $P_4$. The sphere $S_3$ splits $S(\bx_2)$ into two hemispheres, and for $c_2$ to be critical, each of these hemispheres must contain at least one point of $\bx_2$. Note that more than half of the volume of one of these hemispheres lies  inside $\cB(\bx_1)$. We denote this hemisphere by $\hat{S}_3$. Since we assume $m_1=0$, the region $\hat{S}_3\cap\cB(\bx_1)$ does not contain a point of $\bx_2$, hence, there must be a point in $\hat{S}_3{\setminus}\cB(\bx_1)$ (see Figure \ref{fig:Ij_2_case_II}).

\begin{figure}
    \centering
    \begin{subfloat}[3D]
    {\centering
    \includegraphics[width=0.45\linewidth]{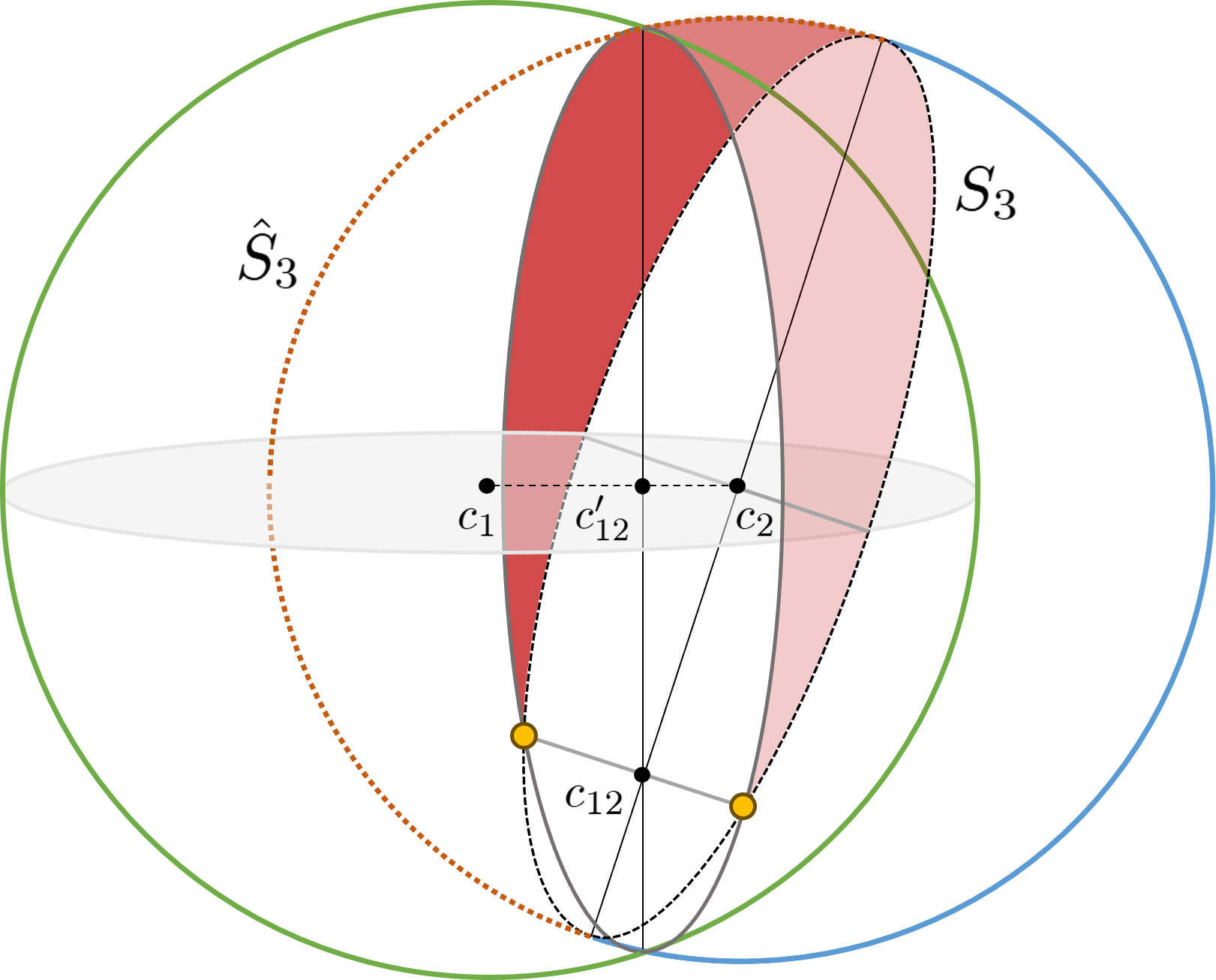}}
    \end{subfloat}
    \hfill
    \begin{subfloat}[Side view]
    {\centering
    \includegraphics[width=0.45\linewidth]{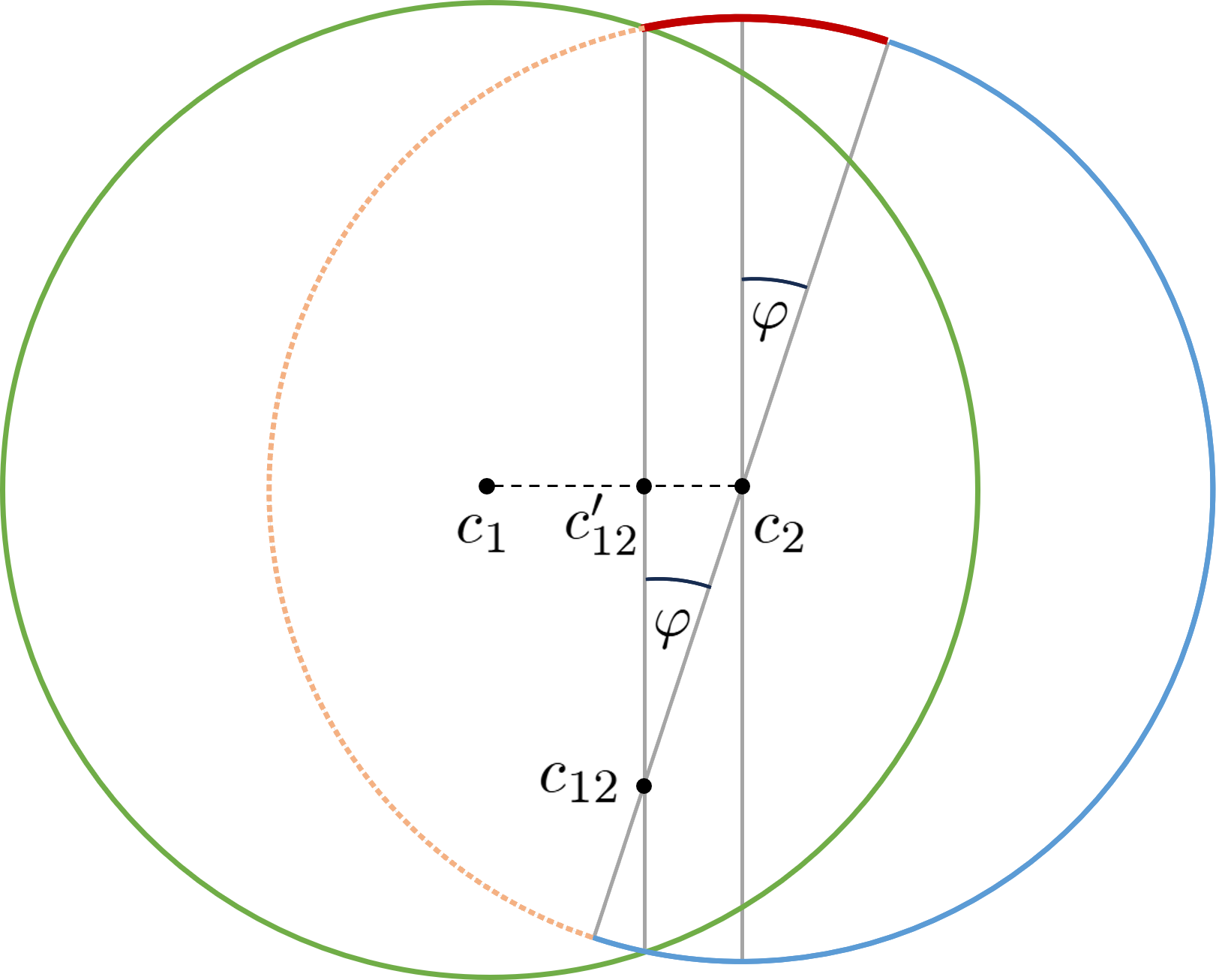}
    }
    \end{subfloat}
    \caption{$I_j^{(2)}$ for $m_1=0$ in $\R^3$. (a) The circle ($1$-sphere) $S_3$ (black dashed line) splits the critical sphere of $c_2$ (in blue) to two hemispheres, one of them is $\hat{S}_3$ (brown dashed line). $c_2$ critical implies than one of the points of the associated configuration must lie in the red region defined by the part of $\hat{S}_3$ not contained in $\cB(\bx_1)$ (the green sphere). (b) Side view of the critical configuration. 
    }
    \label{fig:Ij_2_case_II}
\end{figure}

To compute $\vol(\hat{S}_3{\setminus}\cB(\bx_1))$, denote $S_{12}=S(\bx_1){\cap}S(\bx_2)$ ($(d-2)$-dimensional sphere), and denote by $c_{12}'$ the center of $S_{12}$. Consider the triangle formed by $c_2,c_{12}$ and $c_{12}'$, and denote by $\varphi$ the angle formed by the lines $[c_{12},c_2]$ and $[c_{12},c_{12}']$ (see Figure \ref{fig:Ij_2_case_II}). Then,
\[
\varphi = \arcsin\left(\frac{|c_2-c_{12}'|}{|c_2-c_{12}|}\right)
\leq 
\arcsin\left(\frac{\delta_j}{\epsilon_j}\right) = O\left(\frac{\delta_j}{\epsilon_j}\right).
\]
Therefore, we have
\[
\vol(\hat{S}_3{\setminus}\cB(\bx_1)) = O\left(\frac{\delta_j}{\epsilon_j}\right)\vol(S(\bx_2)).
\]
Plugging this into the integral, and applying Lemma \ref{lem:gamma}, we get
\begin{equation}\label{eq:I_j^(2)-case2}
    \begin{split}
        I_j^{(2)} 
        &
        \leq
        \left(\frac{\delta_{j}}{\epsilon_j}\right)
                \delta_j^{d+1-j}n^{-(2d+1-j)}\Lambda^{2d-j}e^{-\Lambda}
        \sum_{k_{1},k_{2},k_{12}}
        C^*
        (\delta_j\Lambda)^{k_2}
        \Lambda^{k_1+k_{12}}  
        \\
        &
        \leq
        C^*\left(\frac{\delta_{j}}{\epsilon_j}\right)\delta_j^{d-j+k}n^{-(2d+1-j)}\Lambda^{2d-j+2k-2}e^{-\Lambda},
        \end{split}
\end{equation}
where in the last inequality we set $k_1=k_2=k-1$ (recall that $\delta_j\Lambda\rightarrow\infty$).

Note that for $1\leq j\leq d-1$ we have that \eqref{eq:I_j^(2)-case2} upperbounds \eqref{eq:I_j^(2)-case1}, namely,
\[
\delta_j^{d-j+k}n^{-(2d+1-j)}\Lambda^{2d-j+2k-3}e^{-\Lambda}
\leq
\left(\frac{\delta_{j}}{\epsilon_j}\right)\delta_j^{d-j+k}n^{-(2d+1-j)}\Lambda^{2d-j+2k-2}e^{-\Lambda}.
\]
By canceling the identical terms from both sides we get
$1\leq\left(\frac{\delta_{j}}{\epsilon_j}\right)\Lambda$ which (for large enough $n$) is true, since we assume $\delta_j\Lambda\rightarrow\infty$ and $\eps_j\to 0$. For $j=d$, recall that only \eqref{eq:I_j^(2)-case1} holds. Hence, by combining \eqref{eq:I_j^(2)-case1} and \eqref{eq:I_j^(2)-case2}, we get for all $1\leq j\leq d$,
\begin{equation}\label{eq:I_j^(2)-tot}
\begin{split}
I_j^{(2)}
&
\leq         
C^*\left(\frac{\delta_{j}}{\epsilon_j}\right)^{1-\alpha_j}\delta_j^{d-j+k}n^{-(2d+1-j)}\Lambda^{2d-j+2k-2-\alpha_j}e^{-\Lambda},
\end{split}
\end{equation}
where $\alpha_j\coloneqq\ind\{j=d\}$.

\vspace{5pt}
\noindent\textbf{Bounding $I_j^{(3)}$.} 
We separate this case into two regions, 
\begin{equation}\label{eq:I_j^(3,1)-split}
I_j^{(3)} = I_j^{(3,1)} + I_j^{(3,2)},   
\end{equation}
where 
\begin{equation*}
\begin{split}
    I_j^{(3,1)} & \coloneqq \int_{\cA_j'} 
    p(\bx)\ind\{|c_1-c_{12}|>\epsilon_j\rho_1, |c_1-c_2|>\delta_j\rho_1, \rho_2^2\geq \rho_1^2 - |c_1-c_2|^2\}
    d\bx,
    \\
    I_j^{(3,2)} & \coloneqq \int_{\cA_j'} 
    p(\bx)\ind\{|c_1-c_{12}|>\epsilon_j\rho_1, |c_1-c_2|>\delta_j\rho_1, \left(\rho_1-|c_1-c_2|\right)^2\leq\rho_2^2\leq \rho_1^2 - |c_1-c_2|^2\}
    d\bx.
\end{split}    
\end{equation*}

\noindent\underline{Bounding $I_j^{(3,1)}$:}
We start with a lower bound on $V_{\mathrm{uni}}(\bx_1, \bx_2)$.
We have
\[
V_{\mathrm{uni}}(\bx_1, \bx_2) = \omega_d(\rho_1^d+\rho_2^d) - V_{\mathrm{int}}(\bx_1, \bx_2),    
\]
where $V_{\mathrm{int}}(\bx_1, \bx_2)\coloneqq \vol(\cB(\bx_1){\cap}\cB(\bx_2))$.
Since $|c_1-c_2|^2\geq \rho_1^2-\rho_2^2$, and $|c_1-c_2|\geq \delta_j\rho_1$ in this region \eqref{eq:I_j_split}, we apply Lemma \ref{lem:vol_r1_r2}
to get the following bound,
\[
V_{\mathrm{int}}(\bx_1, \bx_2)
\leq
\frac{\omega_{d}}{2}(\rho_{1}^{d} + \rho_{2}^{d}) -
D_{\mathrm{int}}\delta_j(\rho_1^{d} + \rho_2^{d}),
\]
where $D_{\mathrm{int}}>0$, and therefore,
\[
V_{\mathrm{uni}}(\bx_1, \bx_2) 
\geq
\frac{\omega_d}{2}\rho_1^d +
\omega_{d}r^d\left(C_1\delta_j + \frac{1}{2}\right),
\]
where $C_1\coloneqq 2D_{\mathrm{int}}/\omega_d$, and we used $r\leq \rho_1,\rho_2$.
Applying the last bound and the bounds $V_1(\bx),V_2(\bx),V_{12}(\bx)\leq n\omega_d\rho_1^d$, we get
\begin{equation*}
\begin{split}
I_j^{(3,1)}
&
\leq
\sum_{k_{1},k_{2},k_{12}}
C^*n^{k_1+k_2+k_{12}}
e^{-\Lambda(1/2+C_1\delta_j)}
\int_{\cA_j'}
\rho_1^{d(k_1+k_2+k_{12})}e^{-n(\omega_d/2)\rho_1^d}d\bx. 
\end{split}    
\end{equation*}
Next, since $\cB(\bx_1){\cap}\cB(\bx_2)\neq\emptyset$, we have $|c_1-c_2|\leq 2\rho_1$. Thus, the points in $\bx{\setminus}\bx_1$ lie inside a ball of radius $3\rho_1$ around $c_1$, whose volume is $3^d\omega_d\rho_1^d$.
Hence,
\begin{equation*}
\begin{split}
I_j^{(3,1)}
&
\leq 
\sum_{k_{1},k_{2},k_{12}}
C^*n^{k_1+k_2+k_{12}}
e^{-\Lambda(1/2+C_1\delta_j)}
\int_{\bx_1}h_r(\bx_1)\rho_1^{d(k_1+k_2+k_{12})}e^{-n(\omega_d/2)\rho_1^d}
\\
&
\times
\int_{\bx{\setminus}\bx_1}
\ind\{\bx{\setminus}\bx_1\in B_{3\rho_1}(c_1)\}
d\bx
\\
&
=
\sum_{k_{1},k_{2},k_{12}}
C^*n^{k_1+k_2+k_{12}}
e^{-\Lambda(1/2+C_1\delta_j)}
\int_{\bx_1}h_r(\bx_1)\rho_1^{d(d+1-j+k_1+k_2+k_{12})}e^{-n(\omega_d/2)\rho_1^d}
d\bx.
\end{split}
\end{equation*}
Similarly to previous cases, we use a change of coordinates and apply Lemma \ref{lem:gamma},
\begin{equation}\label{eq:I_j^(3,1)}
\begin{split}
I_j^{(3,1)}
&
\leq
n^{-(2d + 1 - j)}\Lambda^{2d-j}e^{-\Lambda-C_1\delta_j\Lambda}
\sum_{k_{1},k_{2},k_{12}}C^*\Lambda^{k_1+k_2+k_{12}}
\\
&
\leq
C^*n^{-(2d + 1 - j)}\Lambda^{2d-j+2k-2}e^{-\Lambda-C_1\delta_j\Lambda},
\end{split}    
\end{equation}
where in the last inequality we set $k_1+k_2+k_{12}=2k-2$ \eqref{eq:k1_k2_k12}.

\vspace{5pt}
\noindent\underline{Bounding $I_j^{(3,2)}$:}\label{par:I_j^(3,2)}
We start with an upper bound on $\rho_1-\rho_2$, this will enable us to bound the volume of $\cB(\bx_1){\setminus}\cB(\bx_2)$ from below.
Recall that we are looking at $\Lambda = \log n + (d+k-2)\log\log n +c(n)$, for $c(n) = o(\log\log n)$. Substituting it into the expected value in Proposition \ref{prop:gen_exp_var_Fkr_gen}, we have
\[
\E(\Ncrit{d}) \sim n\Lambda^{d+k-2}e^{-\Lambda}\sim e^{-c(n)}.
\]
Similarly, taking $r'=(1+\delta)r$, where $\delta \sim\frac{\log\log n}{\log n}$, 
we get
$\E(N_{k,r'}^{d})\rightarrow 0$.
In other words, we only need to consider $r\leq \rho_i\leq (1+\delta)r$, where $i=1,2$. 
This also implies
$\rho_1-\rho_2\leq \delta r$.
Recall from \eqref{eq:I_j_split} that for $I_j^{(3,2)}$, we are interested in the region $(\rho_1-|c_1-c_2|)^2\leq \rho_2^2\leq \rho_1^{2}-|c_1-c_2|^2$.
Hence,
\[
|c_1-c_2|^2\leq \rho_1^2-\rho_2^2 = (\rho_1+\rho_2)(\rho_1-\rho_2)
\leq
2|c_1-c_2| \rho_1^2.
\]
So
\[
\delta_j\rho_1\leq |c_1-c_2|\leq \sqrt{2\delta}\rho_1.
\]

Next, we have 
\[
\vuni(\bx_1,\bx_2) = \omega_d\rho_1^d + \vol(\cB(\bx_2){\setminus}\cB(\bx_1)).
\]
We bound the second term from below in the following way. We have 
\begin{equation*}
\begin{split}
\vol(\cB(\bx_2){\setminus}\cB(\bx_1))
&
\geq
\vol(B_{r}(c_2){\setminus}B_{(1+\delta)r}(c_1)) 
\\
&
=
\vol(B_{r}(c_2-c_1){\setminus}B_{(1+\delta)r}(0) )
\\
&
= 
(1+\delta)^d r^d \vol\left(B_{(1+\delta)^{-1}}\left((c_2-c_1)/(1+\delta)r\right){\setminus}B_{1}(0)\right),
\end{split}
\end{equation*}
where in the inequality we used $r\leq\rho_i\leq (1+\delta)r$, and the equalities are due to shifting and scaling.
Next, note that $|c_2-c_1|>\delta_jr$, and $(1+\delta)^{-1}>1-\delta$. Hence, abusing notation, we have
\begin{equation*}
\begin{split}
\vol(\cB(\bx_2){\setminus}\cB(\bx_1))
&
\geq
(1+\delta)^d r^d \vol\left(B_{1-\delta}(\delta_j/(1+\delta)){\setminus}B_{1}(0)\right)
\\
&
\geq
(1+\delta)^d r^d \vol\left(B_{1-\delta}(\delta_j/2){\setminus}B_{1}(0)\right)
\end{split}    
\end{equation*}
We adopt the notation of Lemma \ref{lem:vdiff}, 
\[
\vdiff(\epsilon,\alpha) \coloneqq \vol(B_{1-\alpha\epsilon}(z_2)){\setminus}
\vol(B_{1}(z_1)),
\]
where $\epsilon=|z_2-z_1|>0$, and $\alpha\in(0,1)$. In our case, we have $\epsilon=\delta_j/2$ and  
$\alpha=2\delta/\delta_j=1/2$, by setting $\delta = \delta_j/4$. Thus, by Lemma \ref{lem:vdiff}, we have
\begin{equation*}
\begin{split}
\vdiff(\delta_j,1/2)
&
\geq 
\delta_j\frac{\omega_{d-1}}{2(d+1)}(1-1/4)^{\frac{d+1}{2}} + o(\delta_j) 
\\
&
=
C_2\delta_j+ o(\delta_j),
\end{split}    
\end{equation*}
where $C_2\coloneqq \frac{\omega_{d-1}}{2(d+1)}(1-1/4)^{\frac{d+1}{2}}$. 
Note that since we assume $\delta_j\rightarrow 0$, for large enough $n$ we have $\delta_j<\sqrt{2\delta}$.
Therefore, we have
\[
n\vuni(\bx_1,\bx_2) 
\geq 
n\omega_d\rho_1^d + n C_2\delta_j (1+\delta)^{d}r^d 
\geq
n\omega_d\rho_1^d + C_3\delta_j \Lambda,
\]
where $C_3\coloneqq C_2/\omega_d$.
By applying the usual bounds $\rho_2\leq\rho_1$, $\vsimp(\bs{\theta}_2)/\vsimp(\bs{\theta}_{12})\leq 1$, and $V_1(\bx),V_2(\bx),V_{12}(\bx)\leq \omega_d\rho_1^d$, we have \begin{equation*}
    \begin{split}
        I_j^{(3,2)} 
        &
        \leq
        \sum_{k_{1},k_{2},k_{12}}C^*n^{k_1+k_2+k_{12}}
        e^{-C_3\Lambda\delta_j}
        \int_{\bx_1} d\bx_1 
        h_r(\bx_1)\rho_1^{(d-1)(d+1-j)+d(k_1+k_2+k_{12})}e^{-n\omega_d\rho_1^d}
        \\
        &
        \times
        \int_{c_2} dc_2\ind\{\delta_j\rho_1<|c_1-c_2|\leq\sqrt{2\delta}\rho_1\}
        \\
        &
        \leq
        \sum_{k_{1},k_{2},k_{12}}C^* n^{k_1+k_2+k_{12}}
        e^{-C_3\Lambda\delta_j}
        \int_{\bx_1} d\bx_1 
        h_r(\bx_1)\rho_1^{d(d+1-j)+d(k_1+k_2+k_{12})}e^{-n\omega_d\rho_1^d},
    \end{split}
\end{equation*}
where we upper bounded the integral over $c_2$ with the volume of a $(d-j+1)$-dimensional ball of radius $\rho_1$. 
By repeating similar steps as in previous cases, we thus have
\begin{equation*}
    \begin{split}
        I_j^{(3,2)} 
        &
        \leq
        n^{-(2d+1-j)}\Lambda^{2d-j}e^{-\Lambda}
        e^{-C_3\Lambda\delta_j}
        \sum_{k_{1},k_{2},k_{12}}C^*
        \Lambda^{k_1+k_2+k_{12}}
    \end{split}
\end{equation*}
Since more than half the sphere $S(\bx_2)$ is included in $\cB(\bx_1)$ in this case, for $c_2$ to be critical, we must have  $m_{1}\geq 1$ \eqref{eq:m1_m2}. Hence, the dominant term in the sum above is the one with $k_1=k-2,k_2=k-1$. Therefore, we have
\begin{equation}\label{eq:I_j^(3,2)}
    \begin{split}
        I_j^{(3,2)} 
        &
        \leq
        C^*
        n^{-(2d+1-j)}
        \Lambda^{2d-j+2k-3}e^{-\Lambda-C_3\Lambda\delta_j}.    
    \end{split}
\end{equation}

By substituting the bounds \eqref{eq:I_j^(3,1)},\eqref{eq:I_j^(3,2)} into \eqref{eq:I_j^(3,1)-split}, we get
\[
I_j^{(3)}\leq
C^*n^{-(2d + 1 - j)}\Lambda^{2d-j+2k-3}e^{-\Lambda}\left(\Lambda e^{-C_1\delta_j\Lambda}
+
e^{-C_3\Lambda\delta_j}
\right).
\]
Setting $C_4\coloneqq\min\{C_1,C_3\}$, we get
\begin{equation}\label{eq:I_j^(3)-tot}
 I_j^{(3)}\leq
C^*n^{-(2d + 1 - j)}\Lambda^{2d-j+2k-2}e^{-\Lambda-C_4\delta_j\Lambda}.
\end{equation}

By substituting \eqref{eq:I_j^(1)},\eqref{eq:I_j^(2)-tot},\eqref{eq:I_j^(3)-tot} into \eqref{eq:I_j}, we have that for $1\leq j\leq d$,
\begin{equation}\label{eq:I_j_bound}
\begin{split}
I_j 
& 
= 
Cn^{2(d+1)-j}(I_j^{(1)} + I_j^{(2)} + I_j^{(3)})
\\
&
\leq
C^*n\Lambda^{d+k-2}e^{-\Lambda}
\left(\epsilon_j^{d+1-j+k}\Lambda^{d -j + k} 
+ 
\left(\frac{\delta_j}{\epsilon_j}\right)^{1-\alpha_j}\delta_j^{d-j+k}\Lambda^{d-j+k-\alpha_j} 
+\Lambda^{d-j+k}
e^{-C_4\delta_j\Lambda}
\right)
.
\end{split}
\end{equation}

\vspace{5pt}
\noindent\textbf{Bounding \headermath{I_0 - \E\{N\}^2}.}
Define $\Phi(\cX_1,\cX_2):=\ind\{\cB(\bx_1){\cap}\cB(\bx_2)=\emptyset\}$,  then 
\begin{equation*}
\begin{split}
I_0
&
= 
\E\curparam{\sum_{\substack{\cX_1,\cX_2\in\cP_n^{(d+1)} \\ 
|\cX_1\cap\cX_2|=0}} g_{r}^{k-1}(\cX_1, \cP_n)g_{r}^{k-1}(\cX_2, \cP_n)}
\\
&
=
\E\curparam{\sum_{\substack{\cX_1,\cX_2\in\cP_n^{(d+1)} \\ 
|\cX_1\cap\cX_2|=0}} g_{r}^{k-1}(\cX_1, \cP_n)g_{r}^{k-1}(\cX_2, \cP_n)
\Phi(\cX_1,\cX_2)}
\\
&
+
\E\curparam{\sum_{\substack{\cX_1,\cX_2\in\cP_n^{(d+1)} \\ 
|\cX_1\cap\cX_2|=0}} g_{r}^{k-1}(\cX_1, \cP_n)g_{r}^{k-1}(\cX_2, \cP_n)\left(1-\Phi(\cX_1,\cX_2)\right)}
\\
&
=
T_1+T_2.
\end{split}    
\end{equation*}
We start by showing that $T_1 - \E\{N\}^2\leq 0$.

Using Theorem \ref{cor:palm}, we have
\begin{equation*}
\begin{split}
\mathbb{E}\{N\}^2 & = \frac{n^{2(d+1)}}{((d+1)!)^2}\mathbb{E}\{g_{r}^{k-1}(\cX_1', \cP_n\cup\cX_1')g_{r}^{k-1}(\cX_2', \cP_n'\cup\cX_2')\} \\
T_1 & = 
\frac{n^{2(d+1)}}{((d+1)!)^2}\mathbb{E}\{g_{r}^{k-1}(\cX_1', \cP_n\cup\cX')g_{r}^{k-1}(\cX_2',\cP_n\cup\cX')\Phi(\cX_1',\cX_2')\},
\end{split}    
\end{equation*}
where $\cX_1',\cX_2'$ are independent sets of $d+1$ i.i.d.~random variables uniformly distributed in $\T^d$, $\cX' = \cX_1'\cup\cX_2'$, and $\cP_n'$ an independent copy of $\cP_n$. Note that when $\cB(\cX_1')\cap \cB(\cX_2') = \emptyset$ we have $g_r^{k-1}(\cX_i', \cP_n\cup\cX') = g_r^{k-1}(\cX_i', \cP_n\cup\cX'_i)$. Therefore,
\begin{equation*}
\begin{split}
T_1- \mathbb{E}\{N\}^2 & =
\frac{n^{2(d+1)}}{((d+1)!)^2}
\Big(
\mathbb{E}\{g_{r}^{k-1}(\cX_1', \cP_n\cup\cX_1')g_{r}^{k-1}(\cX_2',\cP_n\cup\cX_2')\Phi(\cX_1',\cX_2')\}
\\
&
- \mathbb{E}\{g_{r}^{k-1}(\cX_1', \cP_n\cup\cX_1')g_{r}^{k-1}(\cX_2',\cP_n'\cup\cX_2')\Phi(\cX_1',\cX_2')\} 
\\
&
- \mathbb{E}\{g_{r}^{k-1}(\cX_1', \cP_n\cup\cX_1')g_{r}^{k-1}(\cX_2', \cP_n'\cup\cX_2')\left(1-\Phi(\cX_1',\cX_2')\right)\}
\Big)
\\
&
\le
\frac{n^{2(d+1)}}{((d+1)!)^2}\mathbb{E}\{\Delta g\},
\end{split}    
\end{equation*}
where 
\[
\Delta g\coloneqq
\Phi(\cX_1',\cX_2')\left(
g_{r}^{k-1}(\cX_1', \cP_n\cup\cX_1')g_{r}^{k-1}(\cX_2', \cP_n\cup\cX_2')-
g_{r}^{k-1}(\cX_1', \cP_n\cup\cX_1')g_{r}^{k-1}(\cX_2', \cP_n'\cup\cX_2')
\right).
\]
To show that $\mathbb{E}(\Delta g)=0$, we use the conditional expectation $\mathbb{E}_{\cX_1',\cX_2'}\{\cdot\}\coloneqq\mathbb{E}\{\cdot|\cX_1',\cX_2'\}$. Given $\cX_1',\cX_2'$, $\Delta g\neq 0$ implies  $\cB(\cX_1')\cap\cB(\cX_2')=\emptyset$. Using the spatial independence of $\cP_n$, we have
\begin{equation*}
\begin{split}
\mathbb{E}_{\cX_1',\cX_2'} & \{\Phi(\cX_1',\cX_2')
g_{r}^{k-1}(\cX_1', \cP_n\cup\cX_1')g_{r}^{k-1}(\cX_2', \cP_n\cup\cX_1')\}   \\
&
=
\Phi(\cX_1',\cX_2')
\mathbb{E}_{\cX_1',\cX_2'}\{
g_{r}^{k-1}(\cX_1', \cP_n\cup\cX_1')
\}
\mathbb{E}_{\cX_1',\cX_2'}\{
g_{r}^{k-1}(\cX_2', \cP_n\cup\cX_2')
\}
\\
&
=
\Phi(\cX_1',\cX_2')
\mathbb{E}_{\cX_1',\cX_2'}\{
g_{r}^{k-1}(\cX_1', \cP_n\cup\cX_1')
g_{r}^{k-1}(\cX_2', \cP_n'\cup\cX_2')
\},
\end{split}    
\end{equation*}
where the last equality is due to $\cP_n$ and $\cP_n'$ being independent and having the same distribution. Therefore, $\mathbb{E}_{\cX_1',\cX_2'}\{\Delta g\}= 0$, and consequently $\mathbb{E}\{\Delta g\} = 0$.

Next, we bound $T_2$. Firstly, we apply Mecke's formula,
\[
T_2 =
\frac{n^{2(d+1)}}{((d+1)!^2)}
\E\curparam{g_{r}^{k-1}(\cX_1', \cX'\cup\cP_n)g_{r}^{k-1}(\cX_2', \cX'\cup\cP_n)\left(1-\Phi(\cX_1',\cX_2')\right)}.
\]
Next, we separate $T_2$ to two regions
\[
T_2 = 2\frac{n^{2(d+1)}}{((d+1)!^2)}\left(I_0^{(1)}+I_0^{(2)}\right)
\]
such that 
\begin{equation*}
    \begin{split}
        I_0^{(1)} 
        &
        \coloneqq
        \int_{\cA_0'}\ind\curparam{|c_1-c_2|\leq\epsilon_0\rho_1}
        p(\bx)
        d\mathrm{x}
        \\
        I_0^{(2)}
        &
        \coloneqq
        \int_{\cA_0'}\ind\curparam{|c_1-c_2|>\epsilon_0\rho_1}
        p(\bx)
        d\mathrm{x}.
    \end{split}
\end{equation*}

By applying the BP formula (Lemma \ref{lem:bp}) for $\bx_2$ ($\bx_1$ and $\bx_2$ are disjoint) we have
\begin{equation*}
    \begin{split}
        I_0^{(1)} 
        =
        &\sum_{k_{1},k_{2},k_{12}}
        \frac{d!n^{k_1+k_2+k_{12}}}{k_1!k_2!k_{12}!}
        \int_{\bx_1} d\bx_1
        \int_{c_2} dc_2
        \int_{\rho_2}d\rho_2
        \int_{\bs{\theta}_2} d\bs{\theta}_2
        \\
        &
        \times 
        h_r(\bx_1)h_r(\bx_2)
        \ind\curparam{\rho_2\leq\rho_1, |c_1 - c_{2}|\leq \epsilon_0\rho_1}
        \rho_2^{d^2-1}\\
        &\times\vsimp(\bs{\theta}_2)
    V_1(\bx)^{k_1}V_2(\bx)^{k_2}V_{12}(\bx)^{k_{12}} e^{-nV_{\mathrm{uni}}(\bx_1, \bx_2)}.
    \end{split}
\end{equation*}
Similarly to case of $I_j^{(1)}$ \eqref{eq:I_j^(1)}, we apply the bounds $\vsimp(\bs{\theta}_2)\leq 1$, $V_1(\bx),V_{12}(\bx)\leq \omega_d\rho_1^d$, $V_{\mathrm{uni}}(\bx_1,\bx_2)\geq \omega_d\rho_1^d$, and $V_2(\bx)\leq \frac{\kappa_{d}}{2} \epsilon_0\rho_1^{d}$.
Hence,
\begin{equation*}
    \begin{split}
        I_0^{(1)} 
        &
        \leq
        \sum_{k_{1},k_{2},k_{12}}
        C^*\epsilon_{0}^{k_2}
        n^{k_1+k_2+k_{12}}
        \int_{\bx_1} d\bx_1
        h_r(\bx_1)
        \rho_1^{d(k_1+k_2+k_{12})}
        e^{-n\omega_d\rho_1^d}
        \\
        &
        \times
        \int_{c_2}\ind\curparam{|c_1 - c_{2}|\leq \epsilon_0\rho_1} dc_2
        \int_{\rho_2}\rho_2^{d^2-1}\ind\{\rho_2\leq\rho_1\}d\rho_2
        .
    \end{split}
\end{equation*}
Note that the integral over $c_2$ is merely the volume of the ball $B_{\epsilon_0\rho_1}(c_1)$, and thus,
\begin{equation*}
    \begin{split}
        I_0^{(1)} 
        &
        \leq
        \sum_{k_{1},k_{2},k_{12}}
        C^*\epsilon_{0}^{d+k_2}
        n^{k_1+k_2+k_{12}}
        \int_{\bx_1} d\bx_1
        h_r(\bx_1)
        \rho_1^{d(k_1+k_2+k_{12}+1)}
        e^{-n\omega_d\rho_1^d}
        \int_{\rho_2}\rho_2^{d^2-1}\ind\{\rho_2\leq\rho_1\}d\rho_2
        .
    \end{split}
\end{equation*}

Next, we split this case into two subcases according to $k_1$ \eqref{eq:k1_k2_k12}, (1) when $k_1=k-1-k_{12}$ and (2) $k_1<k-1-k_{12}$. 

\vspace{5pt}
\noindent\underline{The case $k_1=k-1-k_{12}$:} In this case we have $m_1=0$ \eqref{eq:m1_m2}. Thus, $c_2$ critical implies $\rho_2\geq \rho_1(1-\epsilon_0^2)^{1/2}$ (otherwise $\cB(\bx_1)$ includes more than half the sphere $S(\bx_2)$, rendering $c_2$ non-critical). Hence,
\begin{equation*}
    \begin{split}
        I_0^{(1)} 
        &
        \leq
        \sum_{k_{1},k_{2},k_{12}}
        C^*\epsilon_0^{d+k_2} n ^{k_1+k_2+k_{12}}
        \int_{\bx_1} d\bx_1
        h_r(\bx_1)\rho_1^{d(k_1+k_2+k_{12}+1)}
        e^{-n\omega_d\rho_1^d}
        \left(\rho_1^{d^2}-\rho_1^{d^2}(1-\epsilon_0^2)^{d^2/2}\right)
        \\
        &
        \approx
        \sum_{k_{1}, k_{2}, k_{12}}
        C^*\epsilon_0^{d+k_2+2}n^{k_1+k_2+k_{12}}
        \int_{\bx_1} d\bx_1
        h_r(\bx_1)\rho_1^{d(d+k_1+k_2+k_{12}+1)}
        e^{-n\omega_d\rho_1^d},
\end{split}
\end{equation*}
where we used $1-(1-\epsilon^{2})^{d^2/2}=\epsilon^2 d^2/2 + O(\epsilon^4).$
Similarly to previous cases, we apply a change of coordinates and Lemma \ref{lem:gamma}, 
\begin{equation}\label{eq:I_0^(1)-case1}
\begin{split}
        I_0^{(1)} 
        &
        \leq
        \epsilon_0^{d+2} n^{-(2d+1)}\Lambda^{2d}e^{-\Lambda}
        \sum_{k_{1},k_{2},k_{12}}
        C^*\Lambda^{k_1+k_{12}}(\epsilon_0\Lambda)^{k_2}
        \\
        &
        \leq
        C^*\epsilon_0^{d+k+1} n^{-(2d+1)}\Lambda^{2d+2k-2}e^{-\Lambda}.
    \end{split}
\end{equation}

\noindent\underline{The case $k_1<k-1-k_{12}$:} 
Recall that taking $r'=(1+\delta)r$, for $\delta \sim\frac{\log\log n}{\log n}$,
we have
$\E(N_{k,r'}^{d})\rightarrow 0$, implying that we only need to consider $\rho_1(1-\delta)\leq \rho_2\leq\rho_1$. 
Therefore, by repeating similar steps as in the previous case, we get
\begin{equation}\label{eq:I_0^(1)-case2}
    \begin{split}
        I_0^{(1)} 
        &
        \leq
        \epsilon_0^{d}\delta n^{-(2d+1)}\Lambda^{2d}e^{-\Lambda}
        \sum_{k_{1},k_{2},k_{12}}
        C^*\Lambda^{k_1+k_{12}}(\epsilon_0\Lambda)^{k_2}
        \\
        &
        \leq
        C^*\epsilon_0^{d+k} n^{-(2d+1)}\Lambda^{2d+2k-3}e^{-\Lambda}.
    \end{split}
\end{equation}
where in the last inequality we set $\delta=\epsilon_0/2$, $k_1=k-2$, and $k_2=k-1$.

To conclude, combining \eqref{eq:I_0^(1)-case1} and \eqref{eq:I_0^(1)-case2} yields
\[
    \begin{split}
        I_0^{(1)} 
        &
        \leq
        C^*\epsilon_0^{d+k} n^{-(2d+1)}\Lambda^{2d+2k-3}e^{-\Lambda}\left(\epsilon_0\Lambda 
        +
        1
        \right)
        \\
        &
        \leq
        C^*\epsilon_0^{d+k+1} n^{-(2d+1)}\Lambda^{2d+2k-2}e^{-\Lambda},
    \end{split}
\]
since we later choose $\epsilon_0\Lambda\rightarrow\infty$, as $n\rightarrow\infty$.

Bounding $I_0^{(2)}$ follows the same lines as in $I_j^{(3)}$ \eqref{eq:I_j^(3)-tot}. Hence,
\[
\begin{split}
I_0^{(2)} & \leq C^*n^{-(2d+1)}\Lambda^{2d+2k-2}e^{-\Lambda-C_4\epsilon_0\Lambda}.
\end{split}   
\]

\noindent\underline{Bounding \headermath{I_0 - \E\{N\}^2} - conclusion:} 
Putting everything together, we get
\begin{equation}\label{eq:I_0-EN2}
I_0-\E\{N\}^2
\leq
T_2
\leq
C^*n\Lambda^{d+k-2}e^{-\Lambda}
\left(
\epsilon_0^{d+k+1} \Lambda^{d+k}
+
\Lambda^{d+k}e^{-C_4\epsilon_0\Lambda}
\right).
\end{equation}

\noindent
\textbf{Convergence of the variance.}
Recall from Proposition \ref{prop:gen_exp_var_Fkr_gen} that $\E\{N\}\sim n\Lambda^{d+k-2}e^{-\Lambda}$. 
For $1\leq j\leq d$, by \eqref{eq:I_j_bound} we have
\begin{equation}\label{eq:Ij_conv}
\frac{I_j}{\E\{N\}}
\leq
C^*
\left(\epsilon_j^{d+1-j+k}\Lambda^{d -j + k} 
+ 
\left(\frac{\delta_j}{\epsilon_j}\right)^{1-\alpha_j}\delta_j^{d-j+k}\Lambda^{d-j+k-\alpha_j} 
+
\Lambda^{d-j+k}
e^{-C_4\delta_j\Lambda}
\right).    
\end{equation}
By choosing $\epsilon_j=\Lambda^{-\frac{d+1/2-j+k}{d+1-j+k}}$, and
$\delta_j=\frac{d-j+k+1}{C_4}\frac{\log\Lambda}{\Lambda}$,
we get
$\frac{I_j}{\mathbb{E}\{N\}}
\overset{n\rightarrow\infty}{\longrightarrow} 0.
$

Next, by \eqref{eq:I_0-EN2} we have
\begin{equation}\label{eq:I0_conv}
\frac{I_0 - \E\{N\}^2}{\E\{N\}}
\leq
C^*
\left(
\epsilon_0^{d+k+1} \Lambda^{d+k}
+
\Lambda^{d+k}e^{-C_4\epsilon_0\Lambda}
\right).    
\end{equation}
Setting $\epsilon_0 =\Lambda^{-\frac{d+k+1/2}{d+k+1}}$ makes the limit go to zero, which concludes the proof.
\end{proof}

As mentioned above, evaluating $\mathrm{Var}(\Ncrit{\mu})$ for $\mu<d$ introduces more special cases than $\mu=d$, due to the following reason. In the proof above for $\mathrm{Var}(\Ncrit{d})$, we relied on the fact that a critical configuration $\cX$ of index $d$ must satisfy $|\cX|=d+1$ and $\cI(\cX,\cP)=k-1$. However, for a critical configuration of index $\mu<d$, the numbers $|\cX|$ and $\cI(\cX,\cP)$ are not fixed, but must satisfy the constraint $|\cX|+\cI(\cX,\cP)=\mu+k$ (see \eqref{eq:I_and_mu}). Consequently, two interacting critical configurations of index $\mu$ may have different sizes, $|\cX_1|,|\cX_2|$, and $I(\cX_1,\cP),I(\cX_2,\cP)$.  
Formally, denoting $\xi_i=|\cX_i|$ for $i\in\{1,2\}$, the second moment $\E\{(\Ncrit{i})^2\}$ takes the form 
\[
\E\{(\Ncrit{\mu})^2\} = \sum_{\xi_1,\xi_2=\max\{2,\mu+1\}}^{\min\{\mu+k,d+1\}}\sum_{j=0}^{\min\{\xi_1,\xi_2\}}I_j,
\]
where $I_j$ (using the same notations as in \eqref{eq:I_j_palm}) is given by
\[
I_j=\frac{n^{\xi_1+\xi_2-j}}{j!(\xi_1-j)!(\xi_2-j)!}
\E\curparam{g_{r}^{\mu+k-\xi_1}(\cX'_1, \cP_n\cup\cX')g_{r}^{\mu+k-\xi_2}(\cX'_2, \cP_n\cup\cX')}.
\] 
This variability in configuration sizes introduces more special cases, complicating the second moment analysis and making the proof more technical and substantially longer.
However, it can be shown that this variability affects only the constants involved, not the asymptotic rate, which depends on the configuration sizes solely through the value of $\mu+k$ \eqref{eq:I_and_mu} 
(the total number of points that define a critical configuration of index $\mu$). This is regardless of how these points are partitioned between the open ball associated with the configuration \eqref{eq:crit_ball} (the set whose size is $\cI(\cX,\cP)$) and its boundary (the set $\cX$).

\begin{proof}[Proof of Proposition \ref{prop:gen_exp_var_Fkr_gen} - convergence rate]
From \eqref{eq:var_explicit}, we have
\[
\mathrm{Var}(\Ncrit{d}) - \E\{\Ncrit{d}\} = 
\sum_{j=1}^{d}I_j + (I_0 - \E\{\Ncrit{d}\}^2).
\]
From \eqref{eq:Ij_conv} we have
\begin{equation*}
\begin{split}
I_j
&
\leq
C^*n\Lambda^{d+k-2}e^{-\Lambda}
\left(
\epsilon_j^{d+1-j+k}\Lambda^{d -j + k} 
+ 
\left(\frac{\delta_j\Lambda}{\epsilon_j}\right)^{1-\alpha_j}(\delta_j\Lambda)^{d-j+k}\Lambda^{-1} 
+
\Lambda^{d-j+k}e^{-C_4\delta_j\Lambda}
\right)
,        
\end{split}    
\end{equation*}
and by \eqref{eq:I0_conv}, we have
\[
I_0 - \E\{\Ncrit{d}\}^2
\leq
C^*n\Lambda^{d+k-2}e^{-\Lambda}
\left(
\epsilon_0^{d+k+1}\Lambda^{d+k}
+
\Lambda^{d+k} e^{-C_4\epsilon_0\Lambda}
\right).    
\]
By choosing $\epsilon_j = \Lambda^{-\frac{d+1/2-j+k}{d+1-j+k}}$,
and $\delta_j=\frac{(d-j+k+1)}{C_4}\frac{\log\Lambda}{\Lambda}$,
we get
\begin{equation*}
\begin{split}
I_j
&
\leq
C^*n\Lambda^{d+k-2}e^{-\Lambda}\\
&\times
\left(
\Lambda^{-1/2}
+ 
\left(\frac{(d-j+k+1)}{C_4}\log\Lambda\right)^{d-j+k+1-\alpha_j}\Lambda^{\frac{-1/2}{d+1-j+k}-\alpha_j\left(\frac{d+1/2-j+k}{d+1-j+k}\right)}
+ \Lambda^{-1}
\right)
\\
&
=
O\left((\log(\Lambda))^{d-j+k+1}\Lambda^{-\frac{1/2}{d+1-j+k}}\right)
\\
&
=
O\left((\log\log n)^{d+k}(\log n)^{-\frac{1}{2(d+k)}}\right).
\end{split}    
\end{equation*}
Next, taking $\epsilon_0=\frac{d+k+1}{C_4}\frac{\log\Lambda}{\Lambda}$ yields 
\[
I_0 - \E\{\Ncrit{d}\}^2
=
O((\log\Lambda)^{d+k+1}\Lambda^{-1})=
o\left((\log\log n)^{d+k}(\log n)^{-\frac{1}{2(d+k)}}\right),
\]
which concludes the proof.    
\end{proof}

\section*{Acknowledgement}
YR was partially supported by the Israel Science Foundation, grant 2539/17. OB was partially supported by the Israel Science Foundation grant 1965/19, by the EPSRC grants EP/Y008642/1, and EP/Y028872/1, and by the Leverhulme Trust grant RPG-2023-144. Part of this work was done while OB was at the Technion -- Israel Institute of Technology.
\bibliographystyle{plain}
\bibliography{general}
\appendix

\section{Spherical volumes}
The following spherical volumes lemmas play an important role in calculating bounds in the proof of Proposition \ref{prop:gen_exp_var_Fkr_gen}. 

\begin{lem}[Lemma B.2 in \cite{bobrowski2022homological}]\label{lem:vdiff}
Let $x_1,x_2\in \R^d$ be such that $\abs{x_1-x_2}= \eps>0$, and let $\alpha\in(0,1)$. Define
\[
\vdiff(\eps,\alpha):= \vol(B_{1-\alpha\eps}(x_2){\setminus} B_1(x_1) ).
\]
Then,
\[
\vdiff(\eps,\alpha)  \ge \eps \frac{\omega_{d-1}}{d+1} (1-\alpha^2)^{\frac{d+1}{2}} + o(\eps),
\]
and in particular,
\[
\lim_{\eps\to0}\frac{\vdiff(\eps,\alpha)}{\eps} \in (0,\infty).
\]

\end{lem}

\begin{lem}[Lemma B.3 in \cite{bobrowski2022homological}]
\label{lem:vol_r1_r2}
Let $\vint(r_1,r_2,\Delta)$ be the volume of the intersection of balls with radii $r_1,r_2$ and whose centers are at distance $\Delta$ apart. Assuming that $\Delta^2 \ge |r_1^2-r_2^2|$, then there exists $D_\mathrm{int}>0$ such that
\[
	\vint(r_1,r_2,\Delta) \le \frac{\omega_d}{2}(r_1^d+ r_2^d) - D_\mathrm{int}\Delta(r_1^{d-1}+r_2^{d-1}).
\]
\end{lem}
\section{Blaschke-Petkantschin-
type formulae}

In a nutshell, the Blaschke-Petkantschin (BP) formula allows us to evaluate integrals on point configurations  using spherical coordinates on their minimal circumspheres. Recalling our definitions for critical points of the $k$-NN distance function \eqref{eq:crit_ball} this is highly useful for our analysis.
The adaptations of the BP formulae we use are taken from \cite{bobrowski2022homological}, which was based on \cite{edelsbrunner_nikitenko_reitzner_2017}. While originally described for point configurations in $\R^d$, it was shown in \cite{bobrowski2022homological}, that for configurations with diameter smaller than $2\rconv$, the same formulae apply for the torus $\T^d$. 

Let $\bx=(x_1,\ldots,x_{m})\subset(\mathbb{R}^d)^{m}$. 
Assuming the points in $\bx$ are in general position, they lie on a unique 
$(m-2)$-dimensional sphere $S(\bx)$ which itself lies in an $(m-1)$-dimensional hyperplane, denoted $\Pi(\mathbf{x})$. 
The center and radius of $S(\bx)$ are denoted $c(\bx)$ and $\rho(\bx)$, respectively.
Finally, we denote $\bs{\theta}(\mathbf{x})\subset \mathbb{S}^{m-2}$ the spherical coordinates of the points in $\bx$ on $S(\bx)$. We are interested in the bijective transformation $\mathbf{x}\rightarrow (c,\rho,\Pi,\bs{\theta})$.

In the context of this paper, we will be integrating over functions $f:(\mathbb{R}^d)^{m}\rightarrow\R$ that are affine invariant, so that,
\begin{equation}\label{eq:aff_inv}
f(\mathbf{x})=f(c+\rho\bs{\theta}(\Pi)) = f(\rho\bs{\theta}(\Pi_0)) := f(\rho\bs{\theta}),  
\end{equation}
where $\Pi_0$ is the canonical embedding of $\R^{m-1}$ in $\R^d$ as $\R^{m-1}\times\{0\}^k$. In other words, $f$ is independent of $c$ and $\Pi$, but may depend on $\rho$ and $\bs\theta$. Further, we are interested in coordinates in $\T^d$ rather than $\R^d$. However, since our analysis is limited to small neighborhoods, which are locally Euclidean, we can carry on similar calculations. We will implicitly assume that $f(\bx)=0$ whenever the diameter of $\bx$ is more than $2\rconv$.

\begin{lem}[Lemma C.1 in \cite{bobrowski2022homological}]\label{lem:bp}
Let $f:(\mathbb{T}^d)^{m}\rightarrow\R$ be a measurable bounded function
satisfying \eqref{eq:aff_inv}. Then,
\begin{equation*}
\int_{(\T^d)^{m}}f(\bx)d\bx = 
D_{\mathrm{bp}}\int_{0}^{\infty}\int_{(\mathbb{S}^{m-2})^{m}}\rho^{d(m-1)-1}
f(\rho\bs{\theta})(V_{\mathrm{simp}}(\bs{\theta}))^{d-m+2}d\bs{\theta}d\rho,
\end{equation*}
where $V_{\mathrm{simp}}(\bs{\theta})$ is the volume of the $(m-1)$-simplex spanned by $\bs{\theta}$, $D_{\mathrm{bp}}=((m-1)!)^{d-m+2}\Gamma_{d,m-1}$, and $\Gamma_{d,m-1}$ is the volume of the Grassmannian $\mathrm{Gr}(m-1,d)$.

In the case where $f$ is only linear invariant, namely $f(\bx)=f(c+\rho\bth)$, then
\begin{equation*}
\int_{(\T^d)^{m}}f(\bx)d\bx = 
D_{\mathrm{bp}}\int_{\T^d}\int_{0}^{\infty}\int_{(\mathbb{S}^{m-2})^{m}}\rho^{d(m-1)-1}
f(c+\rho\bs{\theta})(V_{\mathrm{simp}}(\bs{\theta}))^{d-m+2}d\bs{\theta}d\rho.
\end{equation*}
\end{lem}
The following lemma presents a partial form of the BP formula. Let $\bx = (x_1,\ldots, x_{k+1}) \in (\T^d)^{k+1}$, and define $\bx_0 = (x_1,\ldots,x_{m+1})$, with its complement $\hat\bx_0 =\bx{\setminus}\bx_0 = (x_{m+2},\ldots,x_{k+1})$. In the following formula, the variables of $\bx_0$ remain fixed, while those in $\hat{\bx}$ undergo a BP-style change of variables. Fix $\bx_0$ and
denote $c_0 = c(\bx_0), \rho_0 = \rho(\bx_0)$. 
Under the assumptions in Section \ref{sec:prelims}, $c(\bx)$ and $\rho(\bx)$ are well-defined such that $\bx_0$ lies on a smaller sphere than that of $\bx$, and therefore 
$c(\bx_0)$ and $\rho(\bx_0)$ are well-defined as well. 

Next, we consider  $\bx$ in the coordinate system $\R^d_{c_0}$, where the origin is set at $c_0$.
Notice that $\bx_0$ spans a $m$-dimensional  plane $\Pi(\bx_0)\subset \R^d_{c_0}$. In that case, the center $c(\bx)$ must lie on the $(d-m)$-dimensional plane orthogonal to $\Pi(\bx_0)$ denoted $\Pi^{^\perp}(\bx_0)$. Once the center $c(\bx)$ is set, in order to determine the $k$-dimensional plane $\Pi(\bx)$, we choose a $\Pi_0^{^\perp}\in \Gr(k-m-1,d-m-1)$, so that $\Pi(\bx) = \Pi(\bx_0 \cup {c}) \oplus \Pi_0^{^\perp} \cong \R^k$. Notice that once we chose $\bx_0$ and $c(\bx)$ then $\rho = \rho(\bx)$ is determined by $\rho = \sqrt{\rho^2_0 + |c-c_0|^2}$. Therefore, in order to determine the location of the points $\hat\bx_0$ all that is remained is to choose their spherical coordinates $\hat\bth_0 \in (\mathbb{S}^{k-1})^{k-m}$. Overall, we obtained a change of variables $\bx \to (\bx_0, c, \Pi_0^{^\perp}, \hat\bth_0)$ which leads to the following statement.

\begin{lem}[Lemma C.3 in \cite{bobrowski2022homological}]\label{lem:bp_partial}
Let $f:(\T^d)^{k+1}\to \R$, be a bounded measurable function. Suppose that $0\le m \le k-1$, then
\[
\int_{\bx} f(\bx)d\bx = \param{\frac{k!}{m!}}^{d-k+1}\int_{\bx_0}d\bx_0
\int_c dc\int_{\Pi_0^{^\perp}} d\Pi_0^{^\perp}\int_{\hat\bth_0} d\hat\bth_0 \frac{\rho^{(d-1)(k-m)}}{|c-c_0|^{(d-k)}}
\param{\frac{\vsimp(\bth)}{\vsimp(\bth_0)}}^{d-k+1}f(\bx),
\]
where $\bx_0 = (x_1,\ldots, x_{m+1})$, $c\in \Pi^{^\perp}(\bx_0) \cong \R^{d-m}$,  $\Pi_0^{^\perp} \in \Gr(k-m-1,d-m-1)$, and $\hat\bth_0 \in (\mathbb{S}^{k-1})^{k-m}$.
In addition, $\bth = (\theta_1,\ldots, \theta_{k+1}) =  \theta(\bx) \in (\mathbb{S}^{k-1})^{k+1}$, $\bth_0 = (\theta_1,\ldots, \theta_{m+1})$, $c_0 = c(\bx_0)$, and $\rho = \rho(\bx) =  \sqrt{\rho^2_0 + |c-c_0|^2}$.
\end{lem}

We conclude this section with a simple calculus statement that 
is extensively used in the proofs above.
\begin{lem} \label{lem:gamma}
Let $0\leq s\leq t$, $m\in\mathbb{N}$, and $\alpha,\beta>0$. Then,
\begin{equation*}
    \int_{s}^{t} \rho^\alpha e^{-\beta\rho^{m}}d\rho =    \frac{1}{m \beta^{(\alpha+1)/m}} \param{\Gamma((\alpha+1)/m, \beta s^m) - \Gamma((\alpha+1)/m, \beta t^m)}
,
\end{equation*}
where $\Gamma(p,t) = \int_t^\infty z^{p-1}e^{-z}dz $ is the upper incomplete gamma function.
\end{lem}

\begin{proof}
We take the  change of variables $z=\beta\rho^m$,
\begin{equation*}
    \begin{split}
    \int_{s}^{t} \rho^\alpha e^{-\beta\rho^{m}}d\rho
    & 
    =
    \frac{1}{m\beta^{(\alpha+1)/m}}\int_{\beta s^m}^{\beta t^m} z^{(\alpha+1)/m-1} e^{-z}d\rho
    \\
    &
    =
    \frac{1}{m \beta^{(\alpha+1)/m}} \param{\Gamma((\alpha+1)/m, \beta s^m) - \Gamma((\alpha+1)/m, \beta t^m)}.
    \end{split}
\end{equation*}
\end{proof}

\section{Mecke's formula} 
The following theorem is a result of Palm theory for Poisson point processes.
\begin{thm}[Mecke's formula] \label{th:Palm}
Let $(X,\rho)$ be a metric space, let $f:X\rightarrow\mathbb{R}$ be a probability density on $X$, and let $\cP_n$ be a random Poisson process on $X$ with intensity $\lambda_n=nf$. Let $h(\cX,\cP_n)$ be a measurable function defined for all finite subsets $\cX\subset\cP_n\subset X$ with $|\cX|=k$. Then
\begin{equation} \label{eq:Mecke}
\E\left\{ \sum\limits_{\substack{\cY\subset \cP_n \\ |\cX|=k}}h(\cX,\cP_n) \right\}=\dfrac{n^k}{k!}\E\{h(\cX',\cX'\cup\cP_n)\},
\end{equation}
where the sum is over all subsets $\cX\subset\cP_n$ of size $|\cX|=k$, and $\cX'$ is a set of $k$ i.i.d.~random variables in $X$ with density $f$, independent of $\cP_n$.
\end{thm}
For a proof of Theorem \ref{th:Palm} see \cite{penrose_random_2003}.

\begin{thm}[Corollary A.2 in \cite{bobrowski2022homological}]\label{cor:palm}
Let $\abs{\cX_1} = \abs{\cX_2} = k$,
\[
\mean{\sum_{ \substack {
                    \cX_1 ,\cX_2\subset \cP_n  \\
                    \abs{\cX_1 \cap \cX_2} = j }}
    h(\cX_1,\cP_n)h(\cX_2,\cP_n)} = {\frac{n^{2k-j}}{j!((k-j)!)^2}} \mean{h(\cX_1',\cX' \cup \cP_n)h(\cX_2',\cX' \cup \cP_n)},
\]
where $\cX_1',\cX'_2$ are sets of $k$ points with $\abs{\cX_1'\cap\cX_2'} = j$, such that  $\cX' := \cX'_1 \cup \cX'_2$ is a set of $2k-j$ iid points in $X$ with density $f$, independent of $\cP_n$.
\end{thm}

\end{document}